\documentclass{article}
\usepackage{geometry} 
\usepackage[utf8]{inputenc}
\usepackage{listings}
\usepackage{amsmath}
\usepackage{amssymb}
\usepackage{amsthm}
\usepackage{mathrsfs}
\usepackage{graphicx}
\usepackage{siunitx}
\usepackage{color}
\usepackage{tcolorbox}
\usepackage{hyperref}
\usepackage{enumitem}
\usepackage{subcaption}
\usepackage[normalem]{ulem}

\newcommand{\rrrr}[1]{}
\definecolor{DarkGreen}{rgb}{0,0.4,0.1}
\definecolor{orange}{rgb}{1,0.4,0}

\newcommand{\ignore}[1]{}

\numberwithin{equation}{section}
\numberwithin{figure}{section}

\newcommand\overmat[2]{%
  \makebox[0pt][l]{$\smash{\overbrace{\phantom{%
    \begin{matrix}#2\end{matrix}}}^{\text{$#1$}}}$}#2}

\newcommand{\reals}{\mathbb{R}}
\newcommand{\integers}{\mathbb{Z}}
\newcommand{\naturals}{\mathbb{N}}

\newcommand*\dint{\mathop{}\!d}
\newcommand{\diff}[1]{\ensuremath{\frac{d^{#1}}{dx^{#1}}}}
\newcommand{\px}{\operatorname{\mathsf{x}}}

\newtheorem{theorem}{Theorem}[section]

\newtheorem{corollary}[theorem]{Corollary}
\newtheorem{proposition}[theorem]{Proposition}
\newtheorem{remark}[theorem]{Remark}

\title{Full operator preconditioning and the accuracy of solving linear systems}

\author{Stephan Mohr\thanks{Department of Physics, Technical University Munich, James-Franck-Str. 1, 85748 Garching, Germany.  { \tt stephan.mohr@tum.de}  }
  \and Yuji Nakatsukasa\thanks{Mathematical Institute, University of Oxford, Woodstock Road,
Oxford, OX2 6GG, UK.  { \tt nakatsukasa@maths.ox.ac.uk}}  
  \and Carolina Urz\'ua-Torres
  \thanks{Delft Institute for Applied Mathematics, Delft University of Technology,
The Netherlands. {\tt c.a.Urzuatorres@tudelft.nl }}
}

\date{}

\begin{document}

\maketitle

\begin{abstract}
Unless special conditions apply, the attempt to solve ill-conditioned systems 
of linear equations with standard numerical methods leads to uncontrollably 
high numerical error. Often, such systems arise from the discretization of 
operator equations with a large number of discrete variables. In this paper we 
show that the accuracy can be improved significantly if the equation is 
transformed before discretization, a process we call full operator preconditioning 
(FOP). It bears many similarities with traditional preconditioning for iterative 
methods but, crucially, transformations are applied at the operator level.
We show that while condition-number improvements from traditional preconditioning 
generally do not improve the accuracy of the solution, FOP can. 
A number of topics in numerical analysis can be interpreted as implicitly 
employing FOP; we highlight (i) Chebyshev interpolation in polynomial approximation, 
and (ii) Olver-Townsend's spectral method, both of which produce solutions of 
dramatically improved accuracy over a naive problem formulation. 
In addition, we propose a FOP preconditioner based on integration for the 
solution of fourth-order differential equations with the finite-element method, 
showing the resulting linear system is well-conditioned regardless of the 
discretization size, and demonstrate its error-reduction capabilities on 
several examples. This work shows that FOP can improve accuracy beyond the 
standard limit for both direct and iterative methods. 

\end{abstract}

\section{Introduction}

Ill-conditioned linear systems $\mathbf{Ax}=\mathbf{b}$ cannot be solved to 
high accuracy: even with a backward stable solution $\mathbf{\hat x}$ satisfying 
$(\mathbf{A}+\Delta \mathbf{A})\mathbf{\hat x}=\mathbf{b}$ with $\|\Delta 
\mathbf{A}\|_2 =O(\epsilon \|\mathbf{A}\|_2)$ where $\epsilon$ is on the order 
of machine precision, it is well known that we only have $\|\mathbf{\hat x}-
\mathbf{x}\|_2/\|\mathbf{x}\|_2\lesssim \epsilon\kappa(\mathbf{A})$, where 
$\kappa(\mathbf{A})=\|\mathbf{A}\|_2\|\mathbf{A}^{-1}\|_2$ denotes the $2$-norm 
matrix condition number~\cite[\S~1.6]{higham}. 
In other words, once the linear system has been set up, there is no way to reduce 
the numerical error, except by going to higher-precision arithmetic or employing 
a symbolic solver \cite[\S~7.3]{greenbaum_book}, \cite[\S~1.3]{higham}. 

In this paper, we present and discuss a method that attempts to circumvent this 
deadlock by taking the only possible route out: Reformulating the problem. The 
method is applicable to systems of linear equations arising from the discretization 
of an equation posed in continuous spaces, where the matrix of coefficients is 
the finite-dimensional representation of a linear operator. Here it is often the 
\emph{discretization size} that determines the conditioning of the linear problem: 
the higher the number of rows and columns of the matrix, the worse its conditioning 
typically becomes, until one obtains garbage or even numerical 
blow-up \cite{hackbusch_book_2}.

Our approach is inspired by the preconditioning of linear systems used for the 
acceleration of iterative methods. It shares the goal of transforming the problem 
into one that is more easily solvable, and---often used as a historical explanation 
for the term ``preconditioning'' \cite{wathen}---it usually aims at reducing the 
condition number. The crucial difference is the level of abstraction at which the 
transformation takes place: Whereas traditional preconditioning applies 
transformation matrices to the potentially ill-conditioned linear system after 
discretization has taken place, the presented method transforms the operator itself 
\emph{before} discretization.
We call our approach \emph{full operator preconditioning} (FOP), emphasizing the 
structural similarities, but indicating that transformations take place on the 
operator level instead of the matrix level, and distinguishing it from what is 
conventionally called operator preconditioning or PDE-inspired preconditioning 
by other authors \cite{preconditioning_equivalent_operators, hiptmair, mardal_2011}, 
which is a form of traditional preconditioning; see 
Section \ref{sec:matrix-preconditioning-does-not-work} for a discussion.

Traditional (matrix) preconditioning has the goal of speeding up an iterative 
method for solving linear systems by  clustering the spectrum of the preconditioned 
matrix, so that a Krylov subspace method converges in a small number of 
iterations~\cite{wathen}. While reducing conditioning is sufficient in the 
symmetric case, it is neither sufficient nor necessary for fast convergence for 
general matrices~\cite[Chapter 3.2]{greenbaum_book}. An underappreciated fact is 
that, unless special structure is present, matrix preconditioning does \emph{not} 
improve the accuracy of the computed solution. FOP, by contrast, does. 

To illustrate the idea of FOP, we give the following rough example:
Let $\mathcal{L}$ be a linear differential operator and $u$ and $f$ functions such 
that
\begin{align}
    \mathcal{L} u = f. \label{eq:prototype-equation}
\end{align}
Precise definitions are given in the next sections.
This equation is normally tackled by choosing an appropriate discretization such 
as \emph{finite differences}~\cite{leveque2007finite}, \emph{finite elements} 
\cite{hackbusch} or spectral methods~\cite{introduction_spectral_methods,trefethen_spectral}, 
which represents the operator $\mathcal{L}$ as a square matrix $\mathbf{L}$. If 
$\mathbf{L}$ is highly ill-conditioned $\kappa(\mathbf{L})\geq \epsilon^{-1}$, then 
computed results could be useless even with an excellent (backward stable) method.

Now assume that there is a solution operator $\mathcal{R}$ inverting the differential 
operator $\mathcal{L}$ with appropriate boundary conditions. By applying $\mathcal{R}$ 
from the left, we transform equation (\ref{eq:prototype-equation}) to
\begin{align}
    \mathcal{I} u = g, \label{eq:prototype-equation-operator-precond}
\end{align}
with the identity operator $\mathcal{I}$ and right-hand side $g = \mathcal{R} f$. 
While the form of equation (\ref{eq:prototype-equation-operator-precond}) may look 
tautologous with the trivial solution $u=g$, it was chosen intentionally to point 
out that it is amenable to the same discretization-plus-numerical-solver strategy 
as the original equation.
But now the operator to be discretized is the identity instead of the differential 
operator $\mathcal{L}$. A reasonable discretization scheme then leads to 
well-conditioned matrices regardless of the the number of terms. 

Equation (\ref{eq:prototype-equation}) is a type of \emph{operator equation}. 
These equations are all tackled in a similar fashion, and examples besides 
differential equations include integral equations and 
interpolation \cite[Ch.~12]{roemisch_book}. All examples considered later 
in this paper arise from operator equations. For this reason, we dedicate 
the first part of Section \ref{chap:basics} to introduce their 
discretization. We also discuss two types of error---numerical and discretization 
error---which are affected differently by changing the discretization. 
We then give our formal definition of FOP, highlighting the contrasts with the 
traditional notion of preconditioning. 
We then explain why FOP can help improve the accuracy in solving ill-conditioned 
linear systems while matrix preconditioning in general cannot, a fact that has 
been treated as a sidenote in the literature \cite[\S~7.3]{greenbaum_book}, and 
is---to our knowledge---first discussed here in full detail. 

As we will see, FOP is already implicitly part of many methods of numerical 
analysis. In the following three sections, we demonstrate the power of the 
full-operator approach with three examples. We start by looking at the classical 
subject of polynomial interpolation of a univariate function.
In Section \ref{sec:interpolation-operator-equation}, 
we formulate the task as the discretization of an operator equation involving 
the identity operator. We interpret changes between different bases of polynomials 
as FOP and show that such a change can lead to system-size-independent conditioning, 
completely removing numerical instabilities.

Section \ref{chap:spectral} continues with a discussion of spectral methods. 
We investigate a method by Olver and Townsend~\cite{olver_townsend}: a change 
from Chebyshev to ultraspherical polynomials leads to a remarkable reduction of 
the condition number and accurate solution. This can be regarded as an application 
of FOP, however, it is not identified as such in the original paper. 

In section \ref{chap:fem}, we turn to finite-element discretizations. Generalizing 
the observations in the previous sections, we design a FOP preconditioner for 
fourth-order differential equations in one dimension. It is based on the idea of 
solving the biharmonic equation algorithmically for the finite-element basis 
functions and, as we show in 
Section~\ref{chap:fem}, it 
reduces the growth of the norm of associated matrices from $\mathcal{O}(n^4)$ 
to $\mathcal{O}(1)$ while also guaranteeing their invertibility. Numerical 
examples illustrate that it allows reduction of the total error below the limit 
imposed by numerical error in the unpreconditioned system.

FOP is also related to the literature on integral equations. By recasting a 
problem (often differential equation) as an integral equation, the resulting 
conditioning of the linear system is often significantly 
better~(e.g. \cite{greengard1991spectral,greengard1991numerical}), leading to 
more accurate solutions. This paper shows that a similar idea can be employed 
in a number of problems in numerical analysis.  

We conclude the paper with a summary, an outlook onto potential topics of further 
research, and a statement about the implications of the topics discussed in this 
paper.

\section{Mathematical basics} \label{chap:basics}

\subsection{Numerical solution of operator equations} 
\label{sec:operator-equations}

Let $\mathcal{L}$ be an operator between two infinite-dimensional Hilbert spaces 
$V$ and $W$,
\begin{align}\nonumber
    \mathcal{L} : V \rightarrow W,
\end{align}
where $\mathcal{L} , V$ and $W$ may encode boundary conditions as appropriate. 

Suppose we are given the equation
\begin{align}
    \mathcal{L} u = f, \label{eq:operator-equation}
\end{align}
where $f \in W$, and we seek the solution $u \in V$. 

Given $n \in \naturals$, we approximate the solution $u$ by a linear combination 
of \emph{trial basis functions} $\{ \phi_i\}_{i=1}^n$
\begin{align}
\bar u:=   \sum_{k=1}^n u_k \phi_k, \quad \text{ with } u_k \in \reals^n, \quad 
k \in \{1, \dotsc, n\} \label{eq:u-linear-combi-trial}
\end{align}
and we call $V_n = \text{span} \{ \phi_i\}_{i=1}^n$ the \emph{trial space}. 
Further, we choose the same number of linearly independent \emph{test basis 
functions} $\{ \psi_i\}_{i=1}^n$, which span the \emph{test space}
$ W_n=\text{span}\{ \psi_i\}_{i=1}^n$. These basis functions are chosen such 
that $V_n \subset V$ and $W_n \subset W$.

Inserting the approximation (\ref{eq:u-linear-combi-trial}) into the operator 
equation (\ref{eq:operator-equation}) and taking the scalar product in $W$ with 
each of the test functions, we obtain a linear system of $n$ equations for the 
same number of unknown coefficients. Assembling the coefficients in the vector 
$\mathbf{u} = \left(u_1, \dotsc, u_n \right) \in \reals^n$, we write the system 
in matrix form
\begin{align}
    \mathbf{L} \mathbf{u} = \mathbf{b}, \label{eq:linear-system-noprec}
\end{align}
where $\mathbf{b}$ and $\mathbf{L}$ have entries
\begin{equation}
  \begin{split}
\mathbf{b}[j] &= (\psi_j, f)_W, \qquad\, \text{ for } j \in \{1, \dotsc, n\},\\
\mathbf{L}[j,k] &= (\psi_j, \mathcal{L} \phi_k )_W, \quad \text{ for } j, k 
\in \{1, \dotsc, n\},    
  \end{split}  \label{eq:linear-system-noprec-matrix}
\end{equation}
and where $(\cdot , \cdot )_W$ is the scalar product defined on the Hilbert space 
$W$. The linear system (\ref{eq:linear-system-noprec}) is then solved with 
an iterative or direct method~\cite{saad_book} for computing a numerical solution 
$\hat{\mathbf{u}}$. 

Solving the original equation (\ref{eq:operator-equation}) this way introduces 
two sources of error: First, there is the error between the \emph{true 
solution $u$ of \eqref{eq:operator-equation}} and its \emph{approximation 
by trial basis functions} $\bar u = \sum_{k=1}^n u_k \phi_k$, where $u_k$ are 
the components of the \emph{exact} solution $\mathbf{u}$ to equation 
(\ref{eq:linear-system-noprec}). We call this the \emph{discretization error} 
$E_D:=\|u - \bar{u}\|$.

The other type of error is the \emph{numerical error} $E_N:=\|\mathbf{u} - \hat{
\mathbf{u}}\|$, which represents the error in solving~\eqref{eq:linear-system-noprec} 
in finite-precision arithmetic to obtain the computed solution $\hat{\mathbf{u}}$, 
and is estimated by $\frac{\Vert \mathbf{u} - \hat{\mathbf{u}} \Vert_2}{\Vert 
\mathbf{u} \Vert_2}= O(\epsilon\kappa(\mathbf{L}))$. 

Since the overall error of a computed solution is roughly the sum $E_D+E_N$ of 
the discretization and numerical errors, to obtain high accuracy we need both to 
be small. A ubiquitous phenomenon in numerical analysis is that while increasing 
the discretization size $n$ usually reduces $E_D$, it also often worsens the 
conditioning of the linear system, thus increasing $E_N$. For small $n$ we always 
have $E_D \gg E_N$; as we increase $n$, at some point $E_N$ becomes the dominant 
term, i.e. $E_N\gg E_D$. Moreover, $E_N$ keeps growing with $n$, resulting in a 
V-shaped accuracy curve with respect to $n$; see e.g.~\cite[Fig.~3.3]{betcke2005reviving} 
and Figures~\ref{fig:fem-biharmonic-nonpreconditioned},\ref{fig:fem-comparison-errors}. 
The goal of FOP is to suppress the growth of $E_N$ and obtain an accuracy curve 
that improves steadily with $n$. 

It is worth noting that in low-order methods such as some finite-difference and 
finite-element methods, it has traditionally been $E_D$ that dominates; numerical 
errors and conditioning therefore appear to have gained little attention in the 
FEM literature.  
However, this may well change: first, when high accuracy is needed, $n$ may need 
to be large enough to enter the regime $E_N\gg E_D$. Second, and more nontrivially, 
the breakeven point where $E_D\approx E_N=O(\epsilon\kappa(A))$ depends on the 
working precision $\epsilon$, and a compelling line of recent research is to use 
low-precision arithmetic for efficiency~\cite{abdelfattah2020survey} in 
scientific computing and data science applications. In such situation, $E_N$ 
would start dominating for a modest discretization size, making FOP an important 
technique to retain good solutions.

\subsection{Matrix-level preconditioning and FOP}
\label{sec:preconditioning-introduction}

As we review in the next section, in addition to larger errors, high condition 
numbers also often result in a long runtime for many iterative methods which are 
chiefly employed for the solution of this type of equation. This makes a reduction 
of $\kappa(\mathbf{L})$ desirable, both for reasons of numerical stability and 
computational speed. Two such reduction methods are contrasted in this paper, 
which we call matrix preconditioning and FOP, respectively. We introduce them now.

One approach is to manipulate the equations after discretization. Traditionally, 
preconditioning involves the definition of suitable matrices $\mathbf{R}_l$ and 
$\mathbf{R}_r \in \reals^{n \times n}$, one potentially the identity, and subsequent 
solution of
\begin{align}
    \mathbf{R}_l \mathbf{L} \mathbf{R}_r \mathbf{v} = \mathbf{R}_l \mathbf{f}, 
    \label{eq:matrix-preconditioning}
\end{align}
for $\mathbf{v} \in \reals^n$. The coefficient vector solving the original problem 
is obtained by
\begin{align}\nonumber
    \mathbf{u} = \mathbf{R}_r \mathbf{v}.
\end{align}
This is what we call \emph{matrix preconditioning}. Note that $\mathbf{R}_l$ and 
$\mathbf{R}_r$ do not necessarily need to be available as matrices; for many 
algorithms it suffices to be able to compute their linear action on a vector 
\cite{wathen}. The term matrix preconditioning instead refers to the fact that 
preconditioning takes place after matrices have already been computed, in contrast
to the next method.

In our main subject of FOP, instead of applying changes after the discretization 
has taken place, we manipulate the equation on the operator level. Let
\begin{align}\nonumber
    \mathcal{R}_r &: \widetilde{V} \rightarrow V, \qquad 
    \mathcal{R}_l : W \rightarrow \widetilde{W}
\end{align}
be linear operators and consider the equation
\begin{align}
    \mathcal{R}_l \mathcal{L} \mathcal{R}_r v = \mathcal{R}_l f. 
    \label{eq:operator-equation-prec}
\end{align}
This is formally identical to the original operator equation (\ref{eq:operator-equation}), 
but with different (potentially better) numerical properties. We now 
solve~\eqref{eq:operator-equation-prec} as before: we choose new test and trial 
spaces
\begin{align}\nonumber
    \widetilde{V}_n = \text{span} \left( \Phi_1, \dotsc, \Phi_n \right) \subset 
    \widetilde{V}, \qquad 
    \widetilde{W}_n = \text{span} \left( \Psi_1, \dotsc, \Psi_n \right) \subset
    \widetilde{W},
\end{align}
and discretize analogously as before to obtain 
\begin{align}
    \tilde{\mathbf{L}} \tilde{\mathbf{u}} = \tilde{\mathbf{b}}, 
    \label{eq:linear-system-operator-prec}
\end{align}
where
\begin{align}\nonumber
  \tilde{\mathbf{L}}[j,k]  = (\Psi_j, \mathcal{R}_l \mathcal{L} \mathcal{R}_r
  \Phi_k)_{\widetilde{W}}, \qquad     
  \tilde{\mathbf{b}}[j] & = (\Psi_j, \mathcal{R}_l f)_{\widetilde{W}}. 
\end{align}
The solution of the original system is then approximated by
$    u = \sum_{k=1}^n \tilde{u}_k \mathcal{R}_r \Phi_k,$
where $\tilde{\mathbf{u}} = (\tilde{u}_1, \dotsc, \tilde{u}_n)$ is the solution 
of equation (\ref{eq:linear-system-operator-prec}).

\subsection{Operator preconditioning and FOP}
It is crucial to distinguish FOP here from what is sometimes called \emph{operator 
preconditioning} or \emph{PDE inspired preconditioning} in the literature 
\cite{preconditioning_equivalent_operators, hiptmair, mardal_2011}.
They propose to find $\mathcal{R}: W \to V$ 
such that $\mathcal{R} \mathcal{L}$ (resp. $\mathcal{L} \mathcal{R}$)
is an endomorphism on the continuous level. Then, each operator is discretized 
\emph{separately}. Under some conditions on the discretization, the resulting 
matrices are guaranteed to fulfill certain properties that make 
the matrix product $\mathbf{R}\mathbf{L}$ (resp. $\mathbf{L}\mathbf{R}$) 
well-conditioned. 
Importantly, in these methods, the system matrix $\mathbf{L}$ is not changed.
Hence, in the classification above, these are examples of matrix preconditioning.

Nevertheless, the operators proposed as continuous models for matrix 
preconditioners in \cite{preconditioning_equivalent_operators} and \cite{mardal_2011} 
lead to very potent FOP preconditioners, and the class of operators considered 
in \cite{preconditioning_equivalent_operators} contains the FOP preconditioners 
used in Sections \ref{chap:spectral} and \ref{chap:fem} of this work, such that 
operator preconditioning and FOP take inspiration from the same source.

FOP can also be understood as a change of the trial and test bases in $V$ and 
$W$, with no additional operators involved. Let $\mathcal{R}_l^* : \tilde{W} 
\rightarrow W$ denote the adjoint of $\mathcal{R}_l$. Then choosing 
\begin{align}\nonumber
    V_n &= \text{span}\{ \mathcal{R}_r \Phi_i\}_{i=1}^n, \qquad W_n = 
    \text{span}\{ \mathcal{R}_l^* \Psi_i\}_{i=1}^n,
\end{align}
as trial and test spaces instead of the original $\text{span}\{ \phi_1, \dotsc, 
\phi_n\}$ and $\text{span}\{ \psi_1, \dotsc, \psi_n \}$ and following the regular 
discretization procedure in (\ref{eq:linear-system-noprec}) with no preconditioning 
leads to the same system as the FOP procedure in (\ref{eq:operator-equation-prec}). 
This equivalent formulation is sometimes helpful for understanding an algorithm 
as an application of FOP or for deriving the form of the operators $\mathcal{R}$.

Independent of the interpretation of the preconditioning, a central requirement 
for FOP, besides the abstract definition of suitable $\mathcal{R}$, is the 
ability to compute elements of the matrix $\tilde{\mathbf{L}}$ to sufficient 
accuracy---in particular an accuracy that is independent of the system size $n$ 
and the condition number $\kappa(\mathbf{L})$ of the original matrix. 
If this is possible, FOP provides a way to decidedly improve the numerical error.
Before we make this statement more specific, 
we discuss the power and limitation of matrix preconditioning.

\subsection{Matrix preconditioning improves speed but not accuracy} 
\label{sec:matrix-preconditioning-does-not-work}
A classical convergence bound for the conjugate gradient (CG) method shows that 
for a positive definite linear system $\mathbf{L}\mathbf{u}=\mathbf{b}$, the 
$\mathbf{L}$-norm error converges exponentially with constant $\frac{\sqrt{\kappa
(\mathbf{L})} - 1}{\sqrt{\kappa(\mathbf{L})} + 1}$. Similar bounds hold for MINRES 
for symmetric indefinite systems~\cite[Chapter 8]{greenbaum_book}, \cite{liesen_2004}. 
Traditional matrix preconditioning thus aims to reduce the condition number, 
thereby speeding up convergence. 
When GMRES is applied to nonsymmetric/nonnormal linear systems, reducing the 
condition number does not necessarily improve speed \cite{greenbaum_1996}; however, 
one could solve the normal equation by CG once the system is well conditioned. 
Krylov subspace methods generally converge rapidly when the spectrum is clustered 
at a small number of points away from 0; some preconditioners aim to achieve 
this \cite{wathen}. 

While a good (matrix) preconditioner can dramatically improve the \emph{speed}, 
an aspect that is often overlooked is that it does not improve the \emph{accuracy} 
of the solution. A brief comment on this is given in \cite[Chapter 7.3]{greenbaum_book}. 

To gain insight, consider the following situation: Let $\mathbf{L}$ be a matrix with 
arbitrary condition number, and suppose we want to solve the equation
\begin{align}
    \mathbf{L} \mathbf{x} = \mathbf{b} \label{eq:matrix-equation}
\end{align}
Suppose also that an effective preconditioner $\mathbf{R}$ is available, so that 
$\mathbf{R} \mathbf{L}$ is close to the identity. Then the condition number of 
$\mathbf{R}$ must be similar to that of $\mathbf{L}$: the matrix $\mathbf{R}$ 
approximates $\mathbf{L}^{-1}$, and $\kappa(\mathbf{L}) = \kappa(\mathbf{L}^{-1})$. 

When solving (\ref{eq:matrix-equation}) with a preconditioned iterative algorithm, 
each iteration involves one multiplication of the current iterate by each $\mathbf{L}$ 
and $\mathbf{R}$, see again \cite[Chapter 8]{greenbaum_book}.
By \cite[(3.12)]{higham}, matrix-vector multiplication on a computer suffers from an 
error proportional to the norm of the matrix and the vector: It holds
\begin{align}
    fl(\mathbf{L} \mathbf{v}) = \mathbf{L} \mathbf{v} + \mathbf{\zeta}, \quad 
    \text{ with } \quad \Vert \mathbf{\zeta} \Vert_2 \leq \epsilon \Vert \mathbf{L} 
    \Vert_2 \Vert \mathbf{v} \Vert_2,\label{eq:higham-matrix-vector-bound2}
\end{align}
where $\epsilon \approx 10^{-16}$ is close to machine precision and $fl( \cdot )$ 
denotes the result of a floating-point computation. Using the same fact again, 
we find
\begin{align}\nonumber
    fl(\mathbf{R} ( \mathbf{L} \mathbf{v} + \mathbf{\zeta})) 
    = \mathbf{R} (\mathbf{L} \mathbf{v} + \mathbf{\zeta}) + \mathbf{\xi}, \quad 
    \text{ with } \quad \Vert \mathbf{\xi} \Vert_2 \leq \epsilon \Vert \mathbf{R} 
    \Vert_2 \Vert \mathbf{L} \mathbf{v} + \mathbf{\zeta} \Vert_2.
\end{align}
Together, this leads to the estimate
\begin{align}
\Vert \mathbf{R} \mathbf{L} \mathbf{v} - fl(\mathbf{R} fl(\mathbf{L} \mathbf{v})) 
\Vert_2 = \Vert \mathbf{R} \mathbf{\zeta} +  \mathbf{\xi} \Vert_2 
\approx 2 \epsilon \Vert \mathbf{R} \Vert_2 \Vert \mathbf{L} \Vert_2 
\Vert \mathbf{v} \Vert_2 \approx 2 \epsilon \kappa(\mathbf{L}) 
\Vert \mathbf{v} \Vert_2.
\label{eq:error-bound-iterative2}
\end{align}

Thus, if $\kappa(\mathbf{L}) \epsilon$ is large, the error introduced in each 
iteration is large as well, and we can not hope to obtain $O(\epsilon)$ accuracy
in any of the iterates. This represents \emph{no improvement} from nonpreconditioned 
iterative methods, for which the optimal residual is in the order of $\epsilon 
\Vert \mathbf{L} \Vert \Vert \mathbf{u} \Vert_2$, see \cite[section 7.3]{greenbaum_book}, 
implying a bound on the relative error from the true solution of $\epsilon 
\kappa(\mathbf{L})$. 

As discussed in the introduction, the $\epsilon \kappa(\mathbf{L})$ error bound 
is also the ``best'' bound with a direct method that gives a backward stable 
solution. Regardless of the preconditioner or numerical method, it is essentially 
impossible to obtain a solution with better than $\epsilon \kappa(\mathbf{L})$ 
accuracy, once the linear system is given. 

It should be noted that these are all worst-case estimates. However, the 
following example as well as the more realistic analyses in Sections 
\ref{sec:interpolation-operator-and-matrix-preconditioning}, 
\ref{sec:spectral-unpreconditioned} and \ref{sec:fem-unpreconditioned} show that 
such bounds give a good sense of the order of magnitude of the actual error.
In Figure \ref{fig:preconditioning-accuracy}, five random matrices with varying 
condition number were generated in $\reals^{n \times n}$ with $n=100$ 
as $\mathbf{L} = \mathbf{U} \mathbf{S} \mathbf{U}^\top$, where $\mathbf{S} = 
\text{diag}\{1, 2, \dotsc, n\}$ and 
$\mathbf{U} \in \reals^{n \times n}$ is a random orthogonal matrix.  
The exact solution was drawn randomly from $\reals^{n}$. The systems were solved 
using three methods: (i) GMRES; (ii) preconditioned GMRES, using the 
inverse of the matrix as a preconditioner, computed via implementing 
$\mathbf{L}^{-1} = \mathbf{U} \mathbf{S}^{-1} \mathbf{U}^\top$; and (iii) 
a direct solver, using LU factorization with pivots. 
With unpreconditioned GMRES, the relative error decreases at first, more rapidly 
for lower condition numbers. It then stagnates at a value (which we call the 
\emph{limiting error}) that is proportional to the condition number of the matrix, 
close to the upper bounds in equation \eqref{eq:error-bound-iterative2}, as 
indicated by ticks of the y-axis. Crucially, the preconditioned GMRES method 
converges in a single step but leads to roughly the same limiting error. The 
relative error of the solution obtained from the direct method lies close to 
this bound as well. 

\begin{figure}[h]
\centering
\includegraphics[width=0.95\textwidth]{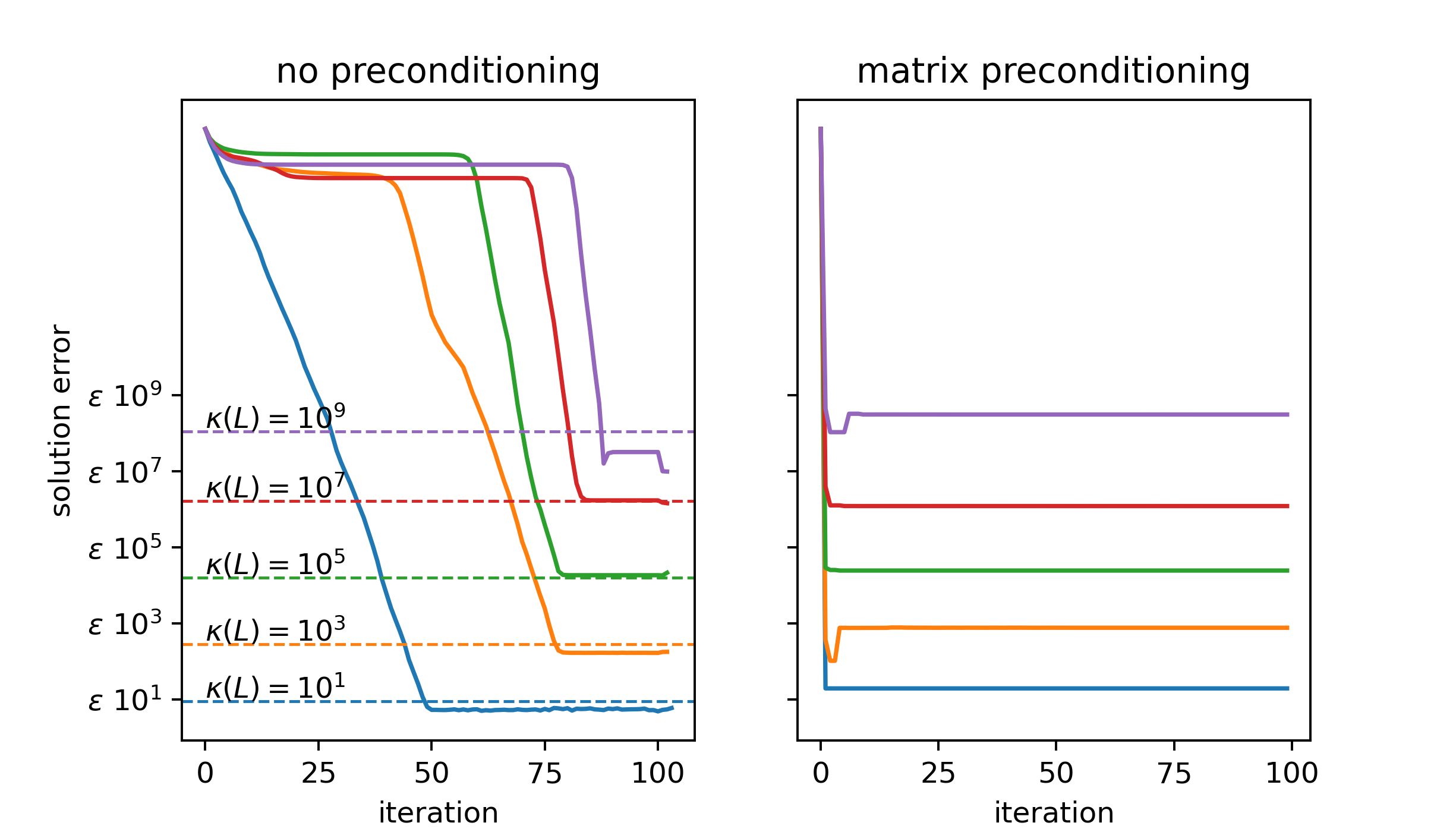}
\caption{Relative error as a function of iteration count for GMRES on linear 
systems with prescribed condition numbers in $\reals^{100 \times 100}$. The 
left panel shows unpreconditioned GMRES, the right preconditioned GMRES with 
$\mathbf{L}^{-1}$ as the preconditioner. Dashed lines in the left panel indicate 
the error of the solution obtained with a direct method. Ticks in the y axis 
are set at $\epsilon \kappa(\mathbf{L})$ for the different used matrices 
$\mathbf{L}$, where $\epsilon = 10^{-16}$ is close to machine precision.}
\label{fig:preconditioning-accuracy}
\end{figure}

This illustrates that the limiting error of the numerical solution roughly scales 
linearly with $\kappa(\mathbf{L})$, no matter what algorithm or preconditioner 
is used. 

One exception to this is when special structure in the matrices can be exploited. 
When tighter bounds than (\ref{eq:higham-matrix-vector-bound2}) hold, matrix 
preconditioning does improve the condition number and error with the same factor.
If, for example, the matrix preconditioner is diagonal, multiplication can be 
performed row or column-wise with machine accuracy, as no summation of elements 
is involved. In fact, diagonal preconditioning is often equivalent to one of 
the simplest forms of FOP, namely, the rescaling of the basis. However, 
unless such special cases apply, matrix preconditioning only improves the speed 
of iterative methods. 

So far in this section, we have seen that ill-conditioned matrices lead to high 
numerical error, and that matrix preconditioning does not alleviate this issue.
This holds even when the matrix preconditioner improves the convergence of 
iterative methods \emph{as if} the preconditioned system has a lower condition 
number. When error reduction is desired, the only effective alternative is to 
remove the ill-conditioning altogether. This is what FOP achieves.

With FOP, it is possible to obtain a solution to the original problem but through 
a different linear system $\tilde{\mathbf{L}} \tilde{\mathbf{u}} = \tilde{
\mathbf{b}}$ with significantly lower condition number $\kappa(\tilde{\mathbf{L}})$.
The system matrix $\tilde{\mathbf{L}}$ can be obtained precisely, as it is not 
the result of some potentially ill-conditioned matrix operation. The system can 
then be solved with improved accuracy (and speed if an iterative solver is used). 
Coming up with a good preconditioner for FOP is not trivial. However, some rules 
can guide this process. The next section introduces the framework which will be 
used to investigate FOP preconditioners. 

It is worth reiterating that to improve accuracy with FOP, the goal should always 
be to reduce the conditioning: If the spectrum is clustered but 
$\kappa(\mathbf{\tilde L})=\kappa(\mathbf{L})$, then only the speed will be 
improved, and not the accuracy.

\subsection{Operator-norm inheritance by discretized matrix} 
\label{sec:desired-properties}

As the condition number is composed of the norm of the matrix and its inverse, 
both factors need to be bounded to control its growth. Bounding the norm of the 
inverse, or even ensuring invertibility at all, is often a complex issue that 
involves additional assumptions in many solution algorithms \cite{boffi2013mixed}. 
Few general statements can be made, and bounds in the following sections are 
established on a case-by-case basis. On the contrary, simple bounds for the norm 
of the matrix are available. They are inherited from the boundedness of the 
operator in the continuous setting. This is known and commonly used in the 
analysis of Galerkin methods, yet often overlooked in other contexts. We therefore 
discuss this here in more generality. 

As a necessary assumption for most standard stability theorems concerning the 
solution of the original equation \cite[Chapter 6]{canuto_quarteroni}, 
\cite[Chapter 1]{finite_elements_for_elliptic_problems}, we assume the operator 
$\mathcal{L} : V \rightarrow W$ to be \emph{continuous}, i.e. there is a 
constant $C_{\mathcal{L}}>0$ such that
\begin{align}\nonumber
    \Vert \mathcal{L} v \Vert_W \leq C_{\mathcal{L}} \Vert v \Vert_V, 
    \quad \text{ for all } v \in V,
\end{align}
where $\Vert \cdot \Vert_W$ and $\Vert \cdot \Vert_V$ are the norms induced 
by the scalar products on $W$ and $V$. The infimum of all such constants 
$C_{\mathcal{L}}$ is called the \emph{norm} of the operator, and we 
denote it by $\Vert \mathcal{L} \Vert$. Note that this norm depends on the 
norms of the spaces $V$ and $W$. Since $\Vert \cdot \Vert_W$ is induced 
by a scalar product, the norm of $\mathcal{L}$ is given by
\begin{align}\nonumber
    \Vert \mathcal{L} \Vert = \sup_{\substack{v \in V \\ v \neq 0}} \frac{\Vert 
    \mathcal{L} v \Vert_W}{\Vert v \Vert_V} = \sup_{\substack{w \in W \\ w \neq 
    0}} \sup_{\substack{v \in V \\ v \neq 0}} \frac{(w, \mathcal{L} v)_W}{\Vert 
    w \Vert_W \Vert v \Vert_V}.
\end{align}

An analogous statement holds for the norm of a matrix. We formulate it here for 
the Euclidean norm on $\reals^n$:
\begin{align}\nonumber
    \Vert \mathbf{L} \Vert_2 = \max_{\substack{\mathbf{v} \in \reals^n}} 
    \frac{\Vert \mathbf{L} \mathbf{v} \Vert_2 }{\Vert \mathbf{v} \Vert_2} = 
    \max_{\substack{\mathbf{w} \in \reals^n \\ \mathbf{w} \neq 0}} 
    \max_{\substack{\mathbf{v} \in \reals^n \\ \mathbf{v} \neq 0}} 
    \frac{\mathbf{w}^\top \mathbf{L} \mathbf{v}}{\Vert \mathbf{w} \Vert_2 
    \Vert \mathbf{v} \Vert_2}.
\end{align}

Inserting the definition (\ref{eq:linear-system-noprec-matrix}) of the 
matrix-representation of $\mathcal{L}$, we obtain
\begin{align}
    \Vert \mathbf{L} \Vert_2 &= \max_{\substack{\mathbf{w} \in \reals^n \\ \mathbf{w} \neq 0}} \max_{\substack{\mathbf{v} \in \reals^n \\ \mathbf{v} \neq 0}} \frac{\mathbf{w}^\top \mathbf{L} \mathbf{v}}{\Vert \mathbf{w} \Vert_2 \Vert \mathbf{v} \Vert_2} \nonumber
= \max_{\substack{\mathbf{w} \in \reals^n \\ \mathbf{w} \neq 0}} \max_{\substack{
\mathbf{v} \in \reals^n \\ \mathbf{v} \neq 0}} \dfrac{\sum_{j=1}^n\sum_{k=1}^n
(w_j \psi_j, \mathcal{L} v_k \phi_k)_W}{\Vert \mathbf{w} \Vert_2 \Vert \mathbf{v} \Vert_2} \nonumber\\
    &\leq \Vert \mathcal{L} \Vert \max_{\substack{\mathbf{w} \in \reals^n \\ \mathbf{w} \neq 0}} \frac{\left\Vert \sum_{j=1}^n w_j \psi_j \right\Vert_W}{\Vert \mathbf{w}_2 \Vert_2} \max_{\substack{\mathbf{v} \in \reals^n \\ \mathbf{v} \neq 0}} \frac{\left\Vert \sum_{k=1}^n v_k \phi_k \right\Vert_V}{\Vert \mathbf{v}_2 \Vert_2} \nonumber\\
    &= \Vert \mathcal{L} \Vert \tilde{\Gamma}_\psi(n) \tilde{\Gamma}_\phi(n). \label{eq:matrix-norm-bound}
\end{align}

Here, we make use of the fact that on $\reals^n$ the norms induced by the bases 
$\phi_j$ and $\psi_k$ are \emph{equivalent} to the Euclidean norm. In other 
words, for each $n \in \naturals$, there are constants $\tilde{\gamma}_\phi(n)
\leq \tilde{\Gamma}_\phi(n)$ and $\tilde{\gamma}_\psi(n) \leq 
\tilde{\Gamma}_\psi(n)$ such that
\begin{align}\nonumber
    \tilde{\gamma}_\phi(n) \Vert \mathbf{v} \Vert_2 \leq \left\Vert \sum_{k=1}^n 
    w_k \phi_k \right\Vert_V \leq \tilde{\Gamma}_\phi(n) \Vert \mathbf{v} \Vert_2,
\end{align}
and analogously for $W$. The equivalence of all norms on a finite-dimensional 
vector space does not mean that these constants do not depend on $n$. In fact, 
by equation (\ref{eq:matrix-norm-bound}), the scaling of $\tilde{\Gamma}_\phi$ 
and $\tilde{\Gamma}_\psi$ with $n$ is the determining factor for the growth of 
the norm of the matrix $\mathbf{L}$, as $\Vert \mathcal{L} \Vert$ does not depend 
on $n$. Note that, in addition to $n$, these constants depend on the choice 
of the basis and on the norm of the spaces $V$ or $W$. Often, results of the form 
$\tilde{\Gamma}_\phi(n) = \Gamma_\phi n^\mu$ are available, where $\mu \in \integers
$, explicitely stating the $n$-dependence of the norm equivalence.

We immediately obtain a desirable criterion for the spaces $V$ and $W$ and for 
the trial and test functions: The test functions must be chosen such that their 
growth is limited in the norms on $V$ and $W$, with respect to which $\mathcal{L}$ 
must be bounded. 

One way to guarantee this criterion is by dividing each basis element by their 
norm. This corresponds to diagonal preconditioning. For some problems, this 
resolves the problem of exploding norms: See for example \cite{bank_scott_1989}, 
who discuss diagonal preconditioning for finite-element methods with highly 
refined meshes. In other cases, such preconditioning negatively impacts the 
norm of the inverse of $\mathbf{L}$, and thus leads to minor or no 
improvements of the condition number. We discuss such an application in 
Section \ref{chap:fem}.

The above process shows the beauty of the full operator approach for 
preconditioning, as important bounds can be derived directly from the operator 
properties.

\section{FOP for polynomial interpolation in $1D$} 
\label{chap:polynomial-interpolation}

We begin our discussion with a very simple and well-known numerical task: the 
\emph{polynomial interpolation} of a function $f$ on the interval $[-1, 1]$ by 
a polynomial $q$. 

In order to formulate this problem, we introduce some notation. For $n \in 
\naturals$, let $\mathbb{P}_{n-1}$ be the space of polynomials of degree up to 
$n-1$, and $\lbrace \px_j \rbrace_{j=1}^n$ be a set of $n$ predetermined 
nodes in $[-1, 1]$. For $q \in \mathbb{P}_{
n-1}$, we require that $q$ interpolates $f$ at $\lbrace \px_j \rbrace_{j=1}^n$:
\begin{align}\label{eq:interpolation-condition}
    q(\px_j) = f(\px_j), \quad \text{ for all } j \in \{1, \dotsc, n\}.
\end{align}

By choosing a basis $\lbrace \phi_i \rbrace_{i=1}^n$ of $\mathbb{P}_{n-1}$, we 
can write the interpolant $q \in \mathbb{P}_{n-1}$ as $q(x) = \sum_{j=1}^{n} q_j 
\phi_j(x)$. Then, plugging this into \eqref{eq:interpolation-condition}, we obtain 
the linear system
\begin{align}
    \underbrace{
    \begin{pmatrix}
        \phi_0(\px_1) & \phi_1(\px_1) & \dotsc & \phi_{n-1}(\px_1) \\
        \phi_0(\px_2) & \ddots & & \phi_{n-1}(\px_2) \\
        \vdots & & \ddots & \vdots \\
        \phi_0(\px_{n}) & \phi_1(\px_{n}) & \dotsc & \phi_{n-1}(\px_{n})
    \end{pmatrix}
    }_{=: \mathbf{L}_{\phi}}
    \begin{pmatrix} q_0 \\ q_1 \\ \vdots \\ q_{n-1} \end{pmatrix}
        = 
    \begin{pmatrix} f(\px_1) \\ f(\px_2) \\ \vdots \\ f(\px_{n}) \end{pmatrix} 
    \label{eq:polynomial-interpolation-linear}
\end{align}
for the coefficients of $q$. We call $\mathbf{L}_{\phi}\in \mathbb{R}^{n\times n}$ 
the \emph{interpolation matrix} (with respect to the basis  $\lbrace \phi_i 
\rbrace_{i=0}^{n-1}$). If the interpolation nodes $\lbrace \px_j \rbrace_{j=1}^n$ 
satisfy $\px_i\neq \px_k$ for $i \neq k$, then 
\eqref{eq:polynomial-interpolation-linear} 
has a unique solution \cite[Theorem 3.1]{atkinson_book}. However, we point out 
that the condition number of the matrix $\mathbf{L}_{\phi}$ determines the 
accuracy and speed with which the linear system can be solved.

A straightforward choice for the basis for $\mathbb{P}_{n-1}$ is the set of 
monomials $\lbrace \mu_i \rbrace_{i=0}^{n-1}$, where $\mu_i(x) := x^{i}, \: i
\in \naturals_+:=\naturals\cup\{0\}$. Under this choice, the matrix $\mathbf{L}_{
\mu}$ is called the \emph{Vandermonde matrix} of the points set $\lbrace \px_j 
\rbrace_{j=1}^n$. Unfortunately, despite the simplicity of its structure, it is 
known for being notoriously hard to solve \cite{Gautschi_1974}. In fact, for 
most choices of $\px_j \in \mathbb{R}$, the condition number of the Vandermonde 
matrix can be shown to grow at least exponentially in $n$ 
\cite{beckermann2000condition}.

Other combinations of basis polynomials and interpolation points can lead to 
better condition numbers. Consider, for instance, the \emph{Chebyshev 
polynomials of the first kind}, defined by
\begin{align}\nonumber
    T_k(x) = \cos(k \arccos(x)), \qquad x \in [-1, 1], \: k\in \naturals_+.
\end{align}
Note that for each $k \in \naturals_+$, $T_k$ is a polynomial of 
degree $k$ \cite[Chapter 3]{trefethen_atap}. Thus, $\lbrace T_k \rbrace_{k=0
}^{n-1}$ forms a basis of the space $\mathbb{P}_{n-1}$. Therefore, we can 
define $\mathbf{L}_{T(n)}$ as the interpolation matrix constructed using 
$\lbrace T_{i} \rbrace_{i=0}^{n-1}$ and the \emph{degree-$n$ Chebyshev nodes}
\begin{align}
    \px_j = \cos\left( \frac{2j - 1}{2(n+1)} \right), \quad \text{ for } 
    j \in \{1, \dotsc, n\},\label{eq:chebpts}
\end{align}
as interpolation nodes. It is known that the matrix $\mathbf{L}_{T(n)}$ is well 
conditioned \cite{strang1999discrete} 
with $\kappa\left(\mathbf{L}_T(n)\right) = \sqrt{2}$ for all $n \in \naturals$.

Although these facts about $\mathbf{L}_{\mu}$ and $\mathbf{L}_{T(n)}$ are 
well known, we believe that FOP brings a new insight 
as to why one system behaves numerically so much better than the other.

\subsection{Polynomial interpolation as discretization of  an operator 
equation} \label{sec:interpolation-operator-equation}

The task of polynomial interpolation can be understood as discretizing the 
operator equation
\begin{align}\nonumber
    \mathcal{I} u = f,
\end{align}
with the identity operator $\mathcal{I}$ and a particular choice of (suitable) 
trial and test spaces $V$ and $W$, respectively. Let $\lbrace \phi_i \rbrace_{
i=0}^{n-1}$ be a basis of $\mathbb{P}_{n-1}$ and define the \emph{trial-space 
basis operator} $\mathcal{X}_\phi : \reals^{n} \rightarrow \mathbb{P}_{n-1}$ as
\begin{align}\nonumber
   \mathcal{X}_\phi \mathbf{u} = \sum_{k=0}^{n-1} u_k \phi_k.
\end{align}
Further, let $\lbrace \px_j \rbrace_{j=1}^n$ be the set of $n$ interpolation 
points. We define the \emph{test operator} $\mathcal{W}_{\px}: W \rightarrow 
\reals^n$ as
\begin{align}\nonumber
    \mathcal{W}_{\px} f = (\delta_{\px_1}(f), \dotsc, \delta_{\px_n}(f))^\top,
\end{align}
where $\delta_{\px}$ is the delta distribution centered at the point $\px$. 
Then, the linear system \eqref{eq:polynomial-interpolation-linear} 
for this discretization can be rewritten as
\begin{align}\label{eq:polynomial-interpolation-discretized}
\mathbf{L}_{\phi} \mathbf{q} = \mathbf{f},
\end{align}
with $\mathbf{f} = \mathcal{W}_{\px} f$, and $\mathbf{L}[i,j]:= (\mathcal{W}_{
\px} \mathcal{I} \mathcal{X}_\phi)[i,j] = \delta_{\px_j}(\phi_i)$ for $i\in 
\{ 0, \dots, n-1 \}, j\in 
\{ 1, \dots, n \}.$

It is clear that choosing a different basis $\lbrace \Phi_i \rbrace_{i=0}^{n-1}$ 
for $\mathbb{P}_{n-1}$ defines a new trial-space basis operator $\mathcal{X}_\Phi$ 
and also a new matrix $\mathbf{L}_\Phi = \mathcal{W}_{\px} \mathcal{I} 
\mathcal{X}_\Phi$.

\subsection{FOP vs. matrix preconditioning} 
\label{sec:interpolation-operator-and-matrix-preconditioning}
Now suppose that $\mathbf{L}_\phi$ is ill-conditioned, whereas  $\mathbf{L}_\Phi$ 
is well-conditioned. Then it is desirable to solve systems involving the matrix 
$\mathbf{L}_\Phi$ rather than $\mathbf{L}_\phi$. This can be achieved by 
right-preconditioning: instead of solving 
\eqref{eq:polynomial-interpolation-discretized}, we work with
\begin{align}\nonumber
\mathbf{L}_\Phi \mathbf{v} = \mathbf{L}_\phi \mathbf{R}_{\Phi}^{\phi} \mathbf{v} 
= \mathbf{f},
\end{align}
where $\mathbf{R}_{\Phi}^{\phi} = \mathcal{X}_\phi^{-1} \mathcal{X}_\Phi \in 
\mathbb{R}^{n\times n}$ is the basis transformation from $\lbrace \Phi_i 
\rbrace_{i=0}^{n-1}$ to $\lbrace \phi_i \rbrace_{i=0}^{n-1}$. One could then 
obtain the original vector $\mathbf{q} = \mathbf{R}_{\Phi}^{\phi} \mathbf{v}$, 
though this could involve an ill-conditioned basis transformation (so it is 
advisable to work with the well-conditioned basis $\Phi$ as much as possible). 

There are two options to implement such right preconditioning. One the one hand, 
one could employ matrix preconditioning, i.e. discretizing first, and multiplying 
the matrices afterward. When the matrices $\mathbf{L}_\phi$ and $\mathbf{R}_{
\Phi}^{\phi}$ are known, this amounts to numerical matrix multiplication 
$\mathbf{L}_\phi \mathbf{R}_{\Phi}^{\phi}$ for direct solution or the application 
of an iterative solver employing one matrix-vector multiplication with both 
$\mathbf{L}_\phi$ and $\mathbf{R}_{\Phi}^{\phi}$ per iteration. On the other hand, 
one could compute the matrix $\mathbf{L}_\Phi$ directly---in this case by evaluating 
the polynomial basis $\Phi_k$ at the interpolation nodes---and employing a 
numerical solver, which avoids matrix multiplication with $\mathbf{L}_\phi$. 
Almost always, the first alternative does nothing to improve the accuracy of the 
final result. 

Let us illustrate this with $\mathbf{L}_\mu$ and $\mathbf{L}_{T(n)}$. For this, 
the matrix $\mathbf{C} := \mathbf{R}_{\mu}^{T(n)}$ can be computed explicitly, 
e.g.~\cite[Ch.~2]{trefethen_atap}. Therefore, a polynomial in 
$\mathbb{P}_{n-1}$ given by its vector of coefficients $\mathbf{u} \in \reals^n$ 
in the Chebyshev basis has coefficients $\mathbf{C}^{\top} \mathbf{u} = 
(\mathcal{X}_{\mu}^{-1} \mathcal{X}_{T(n)}) \mathbf{u}$ in the monomial basis.

We proceed to compare solving $\mathbf{L}_\mu \mathbf{x} = \mathbf{b}$  and 
$\mathbf{L}_{T(n)} \mathbf{x} = \mathbf{b}$ with GMRES. We start by generating 
a function with prescribed coefficients in the monomial basis $\sum_{k=0}^{n-1} 
q_k \phi_k$ in which the coefficients are drawn from the standard normal 
distribution $q_k\sim N(0,1)$, and compute the right-hand side $f(\px_1)$ 
via~\eqref{eq:polynomial-interpolation-linear}, taking $\px_i$ to be the 
Chebyshev nodes~\eqref{eq:chebpts}. We then solve the linear system 
\eqref{eq:polynomial-interpolation-linear} using (a maximum of $n$ steps of) GMRES, 
without and with right preconditioning $\mathbf{C}$. As the focus is to examine 
the best possible accuracy, we ran GMRES with the tightest tolerance: the 
convergence tolerance is set to $\epsilon$ and maximum number of iteration $n$.
For reference we also present the analogous result with (well-conditioned) 
Chebyshev coefficients (without preconditioning), wherein the 'exact' coefficients 
are obtained using Chebfun~\cite{chebfunofficial}. The results are shown in 
Figure~\ref{fig:mcp-error2}.

\begin{figure}[h!] \centering
\includegraphics[width=0.7\textwidth]{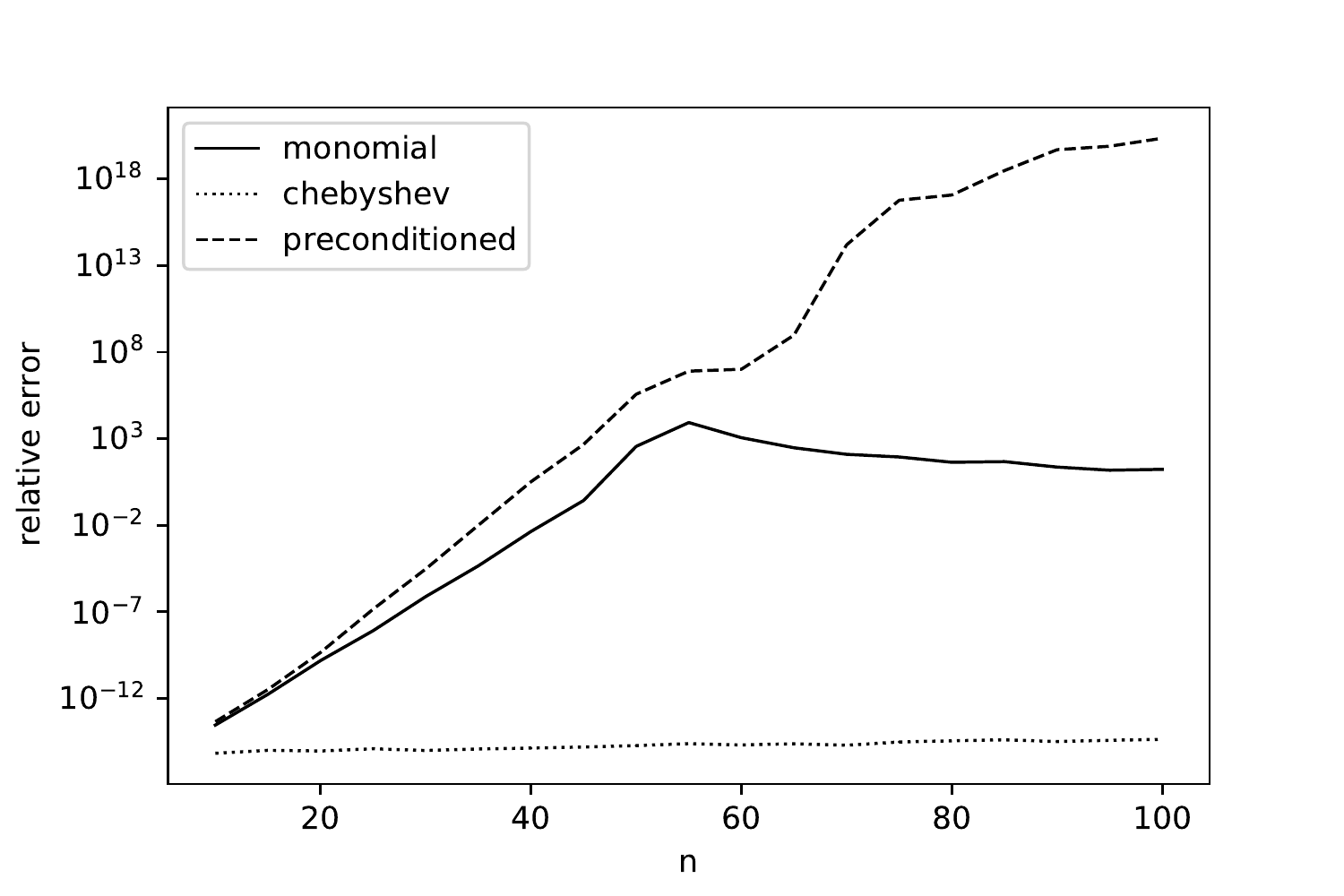}
\caption{\textbf{Relative 2-errors of the computed solutions averaged over 100 draws}. 
Solutions for $\mathbf{L}_\mu \mathbf{x} = \mathbf{b}$ (monomial basis, with and 
without preconditioner $\mathbf{C}$), and $\mathbf{L}_{T(n)} \mathbf{x}_T = 
\mathbf{b}$ (Chebyshev basis) were obtained with GMRES. $\mathbf{x}$ 
and $\mathbf{q}$ are sampled randomly from an $n$-dimensional standard normal 
distribution, and the same right-hand side $\mathbf{b}$ is used for the three 
linear systems. 
}
\label{fig:mcp-error2}
\end{figure}
\begin{figure}[h!] \centering
\includegraphics[width=0.8\textwidth]{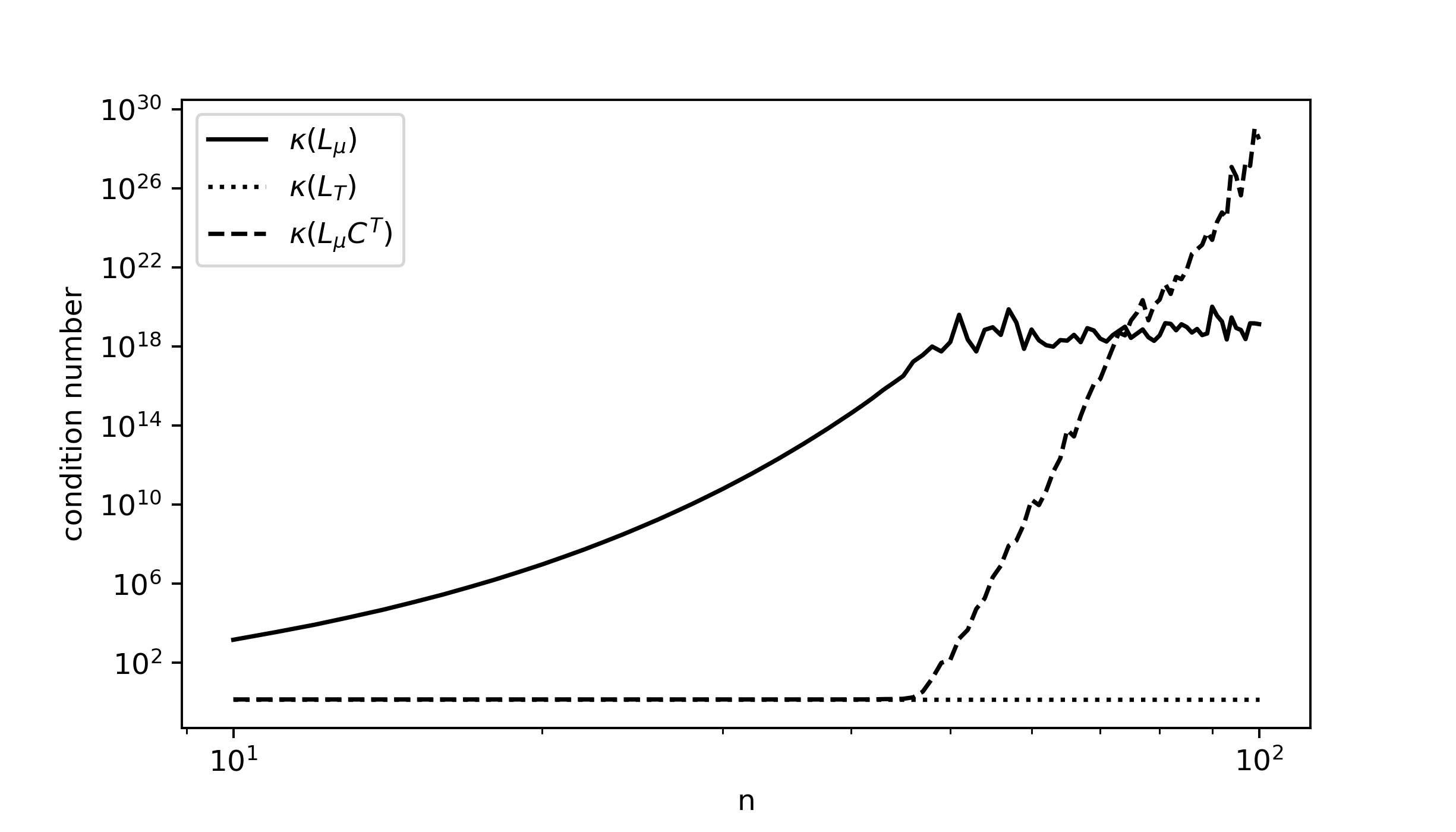}
\caption{Condition numbers of $\mathbf{L}_\mu$, $\mathbf{L}_{T(n)}$ and $\mathbf{L}_\mu 
\mathbf{C}^{\top}$.}
\label{fig:mcp-condition-number}
\end{figure}

As expected, numerical error limits the use of matrix-level preconditioning for 
improving accuracy. Figure \ref{fig:mcp-condition-number} shows the condition 
numbers of the monomial interpolation matrix $\mathbf{L}_\mu$, the Chebyshev 
interpolation matrix $\mathbf{L}_{T(n)}$, and the matrix obtained by multiplying 
$\mathbf{L}_\mu$ and $\mathbf{C}^{\top}$ with finite-precision arithmetic for $n$ 
up to $70$. $\mathbf{C}$ is obtained by the above recursion relation. As expected, 
the condition number of the Vandermonde matrix $\kappa(\mathbf{L}_\mu)$ grows 
exponentially, while $\kappa(\mathbf{L}_{T(n)})$ stays constant. Matrix 
preconditioning is stable until $n\approx 45$, beyond which the associated 
condition number increases unpredictably, even surpassing that of the original 
matrix $\mathbf{L}_\mu$. As the condition number $\kappa(\mathbf{L}_\mu)$ climbs 
to $10^{18}>1/\epsilon$, neither $\kappa(\mathbf{L}_\mu)$ nor the condition number 
of the product with the preconditioner can be expected to be numerically accurate. 
This explains the flattening of $\kappa(\mathbf{L}_\mu)$.

\subsection{Analysis and Discussion}

In order to study the condition numbers of interest, we make use of the fact 
that $\kappa(\mathbf{L}) = \Vert\mathbf{L}\Vert_2 \Vert\mathbf{L}^{-1}\Vert_2$ 
and turn our attention to the norm of the matrices $\mathbf{L}_{\mu}$, 
$\mathbf{L}_{T(n)}$ and their inverses. 
Moreover, we leverage the operator perspective from 
\eqref{eq:polynomial-interpolation-discretized} to find estimates for these 
norms following the spirit of equation \eqref{eq:matrix-norm-bound}. 

First, the Sobolev embedding theorem guarantees that functions 
$f \in H^1(-1, 1)$ are almost everywhere equal to a continuous function. Moreover, 
there is a constant $\alpha > 0$, independent of $f$, such that 
\cite[Chapter 7]{Continuity_Sobolev_Functions}
\begin{align}\label{eq:bound-beta}
    \Vert f \Vert_{L^\infty(-1, 1)} \leq \alpha \Vert f \Vert_{H^1(-1, 1)}, 
    \quad \text{ for all } f \in H^1(-1, 1).
\end{align}

By letting the delta distributions act on the continuous representation of 
functions in $H^1(-1, 1)$, they belong to the space $H^{-1}(-1, 1)$ with norm
\begin{align}
\Vert \delta_{\px_j} \Vert_{H^{-1}(-1, 1)} = \sup_{v \in H^1(-1, 1)}\frac{
\delta_{\px_j}(v)}{\Vert v \Vert_{H^1(-1, 1)}} \leq \sup_{v \in H^1(-1, 1)} 
\frac{\Vert v \Vert_{L^\infty(-1, 1)}}{\Vert v \Vert_{H^1(-1, 1)}} \leq \alpha. 
\label{eq:delta-distributions-continuous}
\end{align}
Then, for any interpolation basis $\lbrace\phi_i \rbrace_{i=1}^n$ in 
$H^1(-1, 1)$, we obtain
\begin{align}\nonumber
\begin{split}
\Vert \mathbf{L}_{\phi} \Vert_2 &= \max_{\mathbf{v} \in \reals^n} 
\max_{\mathbf{w} \in \reals^n} \frac{\mathbf{w}^{\top} \mathbf{L}_{\phi} \mathbf{v}}
{\Vert \mathbf{w}\Vert_2 \Vert \mathbf{v} \Vert_2}
= \max_{\mathbf{v} \in \reals^n} \max_{\mathbf{w} \in \reals^n} \sum_{j, k} 
\frac{w_j v_k \delta_{\px_j}(\phi_k)}{\Vert \mathbf{w}\Vert_2 \Vert \mathbf{v}
\Vert_2} \\
&\leq \max_{\mathbf{v} \in \reals^n} \max_{\mathbf{w} \in \reals^n} \sum_{j, k} 
\frac{w_j v_k \alpha \Vert \phi_k \Vert_{H^1(-1, 1)}}{\Vert \mathbf{w}\Vert_2 
\Vert \mathbf{v}\Vert_2}  
\leq \alpha \max_{1\leq l \leq n} \Vert \phi_l \Vert_{H^1(-1, 1)},
\end{split}
\end{align}
where in the last line, we used \eqref{eq:bound-beta} and the Cauchy-Schwarz 
inequality.

For the monomials in $[-1,1]$, since $\delta(\phi_k) \leq 1$, this estimate is 
easily improved to $\Vert \mathbf{L}_{\mu} \Vert \leq n$, which is still not 
sharp. Thus, the norm of $\mathbf{L}_{\mu}$ is relatively well-behaved. This 
implies that the exponential increase in the condition number of the Vandermonde 
matrix $\kappa(\mathbf{L}_{\mu})$ comes from the contribution of the inverse, 
i.e., the presence of small singular values. 

\begin{remark}
 We observe that the boundedness of the condition number results from two key 
properties:
\begin{itemize}
 \item[(P1)]  the orthogonality of the columns in the interpolation matrix; and 
 \item[(P2)] the controlled behavior of the norms of the columns. \\
\end{itemize}

In view of (P1), one could want to extend the idea of orthogonal-column 
interpolation matrices $\mathbf{L}_{\Phi}$ to other choices of trial 
bases. 
This can easily be done for other orthogonal polynomials, 
see \cite{stephanmscthesis} for a detailed discussion.
\end{remark}

\section{FOP for spectral methods} \label{chap:spectral}

In the next two sections, we give examples of FOP in the context of  differential 
equations. In this section, we focus on spectral methods, and we turn to finite 
element methods in the next.

Spectral methods are known for their excellent convergence properties 
\cite[Chapter 4]{trefethen_spectral}. If the solution of the problem is analytic, 
the error decreases faster than any negative power of the discretization size. 
However, in their most straightforward implementation, spectral methods suffer 
from a fast increase of the condition number, leading to slow convergence and 
numerical instabilities. For high-order differential operators, or if high 
accuracy is desired, this quickly becomes prohibitive. The combination of these 
properties makes spectral methods ideal candidates for FOP.

Several preconditioning techniques have been presented in the literature. Basic 
examples include low-order finite element or finite difference preconditioners, 
or spectral discretizations of the differential operator with variable 
coefficients replaced by constants \cite[\S~4.4]{canuto_quarteroni}. These methods 
rely on linear operations which can only be computed with finite precision. As 
discussed in Section \ref{sec:matrix-preconditioning-does-not-work} and exemplified 
in Section \ref{sec:interpolation-operator-and-matrix-preconditioning}, 
preconditioning only improves accuracy if no multiplication between 
ill-conditioned matrices takes place.

One family of methods that achieve accuracy improvement is known as integration 
preconditioning \cite{olver_townsend}. These methods use relations between the 
spectral basis and the derivatives of its elements, which under certain conditions 
form orthogonal global bases on their own. An early version of integration 
preconditioning was presented by Clenshaw~\cite{clenshaw_1957}, and the method 
was later developed in 
\cite{integration_spectral_tau_proceedings, integration_spectral_tau_nasa, integral_preconditioning_2}. 
We focus here on one particular realization given by Olver and Townsend \cite{olver_townsend}, 
which uses the relationship between Chebyshev and ultraspherical polynomials.
We add a new viewpoint to the analysis by explicitly formulating the operators 
that are used for FOP. This allows us to identify them as generalized integral 
operators and to show they meet the desired criteria laid out in 
Section \ref{sec:desired-properties}.

\subsection{Unpreconditioned spectral methods} 
\label{sec:spectral-unpreconditioned}

Let $\mathcal{L}$ be a linear differential operator of order $N$. As before, we 
limit ourselves to the interval $[-1, 1]$. We assume the leading-order coefficient 
to be non-singular, so that without limitation of generality we may write the 
operator in normalized form
\begin{align}
    \label{eq:spectral-definition-differential-operator}
    \mathcal{L} = \diff{N} + a_{N-1} \diff{N-1} + \dots + a_1 \diff{1} + a_0,
\end{align}
with continuous functions $a_{N-1}, \dots, a_0 : [-1, 1] \rightarrow \reals$.

We want to solve the problem
\begin{align}\nonumber
\begin{split}
\mathcal{L} u = f, \quad \text{ for } x \in (-1, 1), \quad \text{ such that } 
\mathcal{B} u = \mathbf{c},
\end{split}
\end{align}
where $f : [-1, 1] \rightarrow \reals$, $\mathbf{c} \in \reals^N$ and $\mathcal{B}$ 
is a linear operator imposing $N$ linearly independent (boundary) constraints 
on the solution $u$.

Choosing Chebyshev polynomials $\phi_k = T_k$ as the trial basis, we search for 
an approximate  solution in the finite-dimensional space $V_n = \text{span}\{
\phi_i\}_{i=0}^{n-1}$. As the trial basis, projection onto the functions $T_k / 
\Vert T_k \Vert_{L^2(\rho)}^2$ is common, which decomposes the result of the 
application of the differential operator in terms of the Chebyshev polynomials 
and leads to the representation matrix
\begin{align}\nonumber
    \mathbf{L}[i,j] = \frac{(\phi_i, \mathcal{L} \phi_j)_{L^2(\rho)}}{\Vert 
    \phi_i \Vert_{L^2(\rho)}^2}.
\end{align}

For this choice, differentiation is represented by a dense upper-triangular 
matrix $\mathbf{D}$, found for example in \cite{introduction_spectral_methods}. 
Further, due to the convolution formula for Chebyshev polynomials
\begin{align}
    2 T_m T_{k} = T_{m+k} + T_{|m - k|}, \label{eq:chebyshev-convolution}
\end{align}
multiplication with polynomial coefficient functions $a_j$ is replaced by 
multiplication with banded matrices $\mathbf{M}(a_j)$ with bandwidth $2 
\text{ deg}(a_j) + 1$ \cite{integration_spectral_tau_proceedings}. Nonpolynomial 
functions are first expanded in terms of Chebyshev polynomials up to machine 
accuracy and then converted into matrix form.

To make the solution of the system of equations unique, \emph{boundary conditions} 
need to be incorporated. For this, we make use of the method of boundary bordering: 
The last $N$ rows of the $n \times n$ matrix $\mathbf{L}$ are omitted and replaced, 
conventionally swapped to the top of the linear system, by the $N$ linearly 
independent equations coming from the application of $\mathcal{B}$ to the 
approximation $u = \sum u_j \phi_j$. We denote the matrix $\mathbf{L}$ with the 
last $N$ rows left out as $\mathbf{L}_{[N]}$, such that we obtain the system 
$\mathbf{A} \mathbf{u} = \mathbf{b}$ with\footnote{Note that only in this section 
we use $\mathbf{A}$ instead of $\mathbf{L}$ for the coefficient matrix; this is 
because $\mathbf{A}$ contains rows that explicitly reflect the boundary 
conditions, in addition to the discretized operator $\mathcal{L}$. Elsewhere, 
boundary conditions are not explicitly in $\mathbf{L}$, either because they are 
not present or because the basis functions satisfy them.}
\begin{align}\nonumber
\mathbf{A} := 
\begin{pmatrix} 
    \mathcal{B} \phi_0 & \mathcal{B} \phi_1 & \dots & \mathcal{B} \phi_{n-1}  \\
    \mathbf{L}_{[N]} 
\end{pmatrix},
\: \text{ and } \: \mathbf{b}:=
    \begin{pmatrix}
        \mathbf{c} \\
        (\psi_0, f) / (\psi_0, \psi_0) \\
        \hdots \\
        (\psi_{n-N-1}, f) / (\psi_{n-N-1}, \psi_{n-N-1}) 
    \end{pmatrix}.
\end{align}

As a specific example, consider the differential equation
\begin{align}
    \diff{2} u + 10 \diff{} u + 100 \, x \, u = f, \quad \text{ such that } u(\pm 1) = 0.
    \label{eq:example-spectral}
\end{align}

\begin{figure}
\centering
\includegraphics[width=0.9\textwidth]{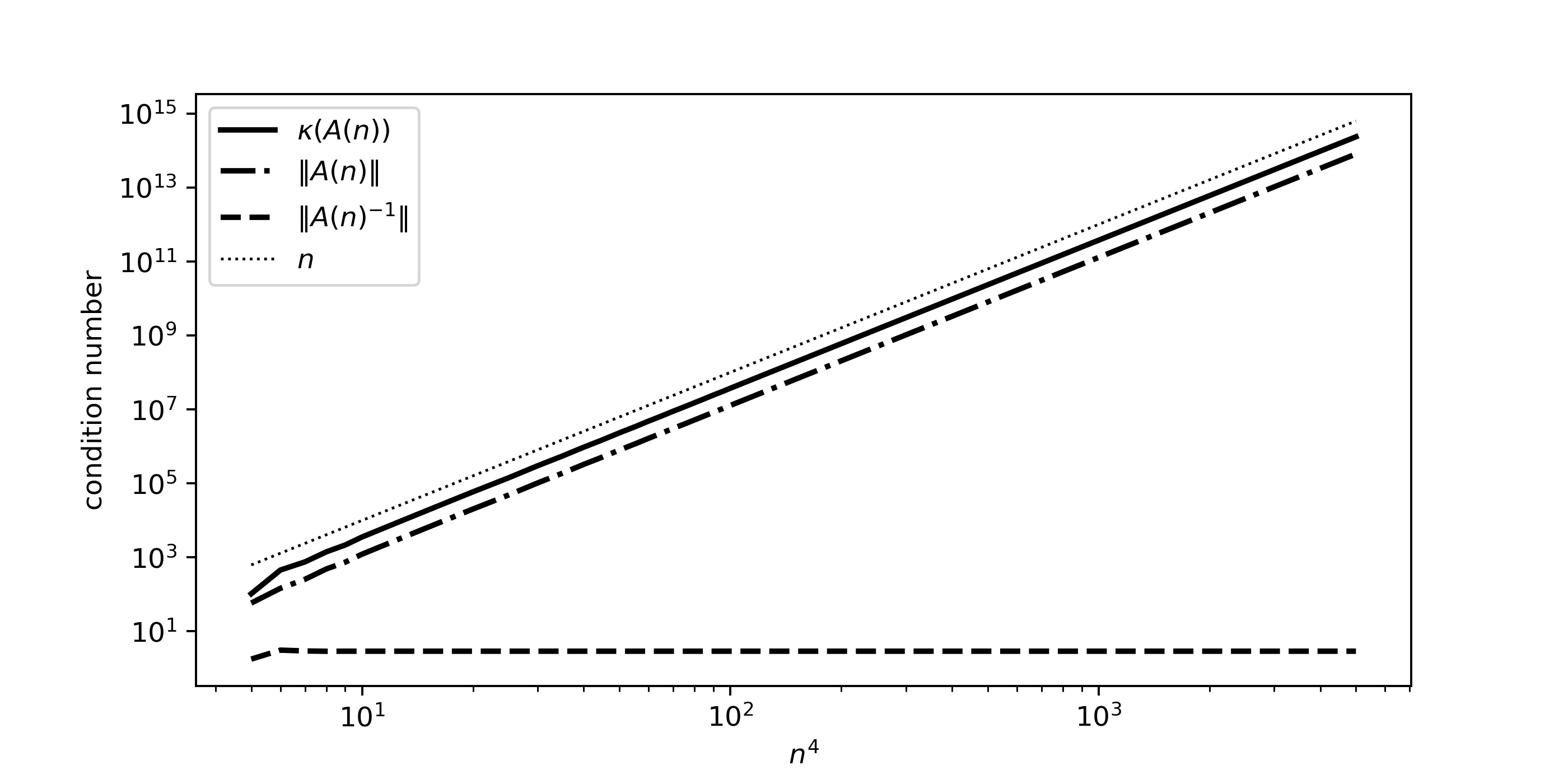}
\caption{Condition number $\kappa(\mathbf{A})$ and norms $\Vert \mathbf{A} 
\Vert_{2}$ and $\Vert \mathbf{A}^{-1}\Vert_{2}$ of the spectral representation 
matrix with Chebyshev basis for the differential equation \eqref{eq:example-spectral} 
with Dirichlet boundary conditions.}
\label{fig:example-spectral-condition-number}
\end{figure}

Figure \ref{fig:example-spectral-condition-number} shows the condition number of 
$\mathbf{A}$ as a function of the discretization size. We confirm that it increases 
as $\mathcal{O}(n^{2N})$ for a differential operator of order $N$ as predicted for 
Chebyshev polynomials \cite[\S~7.7]{boyd_book}.
Also displayed in the figure are the $2$-norms of the matrix $\mathbf{A}$ and of 
its inverse. While $\Vert \mathbf{A}^{-1} \Vert_{2}$ is nearly constant, the 
factor $\Vert \mathbf{A} \Vert_{2}$ is responsible for the rapid growth of the 
condition number $\kappa(\mathbf{A}) = \Vert \mathbf{A} \Vert_{2} \Vert \mathbf{
A}^{-1} \Vert_{2}$. Far from conclusive for general differential equations, this 
example shows that bounds on the norm of the matrix $\mathbf{L}$, such as the 
bound (\ref{eq:matrix-norm-bound}) inferred from the continuity of the operator, 
can be of use in the development and understanding of preconditioners.

\subsection{Ultraspherical polynomials} \label{sec:olver_townsend}

As already shown in Section \ref{chap:polynomial-interpolation}, structural 
properties of matrices representing the same operator but obtained with different 
trial and test bases may differ fundamentally. The method presented in 
\cite{olver_townsend} is a beautiful example of this. Retaining Chebyshev 
polynomials as the trial-space basis, they switch the test-space basis to 
ultraspherical polynomials, thereby reducing the growth of the condition number 
from $\mathcal{O}(n^{2N})$ to $\mathcal{O}(n)$. 

Futher, this may be seen as an application of FOP:
The three operations involved in constructing the spectral matrix 
$\mathbf{L}$---differentiation, multiplication and basis change---may be 
computed by a recursion. This allows the computation of the preconditioned 
matrix without any ill-conditioned matrix multiplication.

For $\lambda \in \naturals$, the \emph{ultraspherical polynomials} 
$(C^{(\lambda)}_k)_{k \in \naturals}$ of order $\lambda$ are defined as the 
family of polynomials orthogonal with respect to the $L^2$-scalar product 
with weight
\begin{align}\nonumber
    \rho^{(\lambda)}(x) = (1-x^2)^{\lambda - 1/2}, \quad x \in (- 1, 1)
\end{align}
and normalized such that
\begin{align}\nonumber
    C^{(\lambda)}_k(x) = \frac{2^k (\lambda + k - 1)!}{(\lambda-1)! k!} x^k 
    + \mathcal{O}(x^{k-1}).
\end{align}
Together with the Chebyshev polynomials, they fulfill the defining properties
\begin{align}
    \label{eq:ultraspherical-derivative}
    \frac{dC^{(\lambda)}}{dx} = \begin{cases}
        2 \lambda C^{(\lambda + 1)}_{k-1}, \quad &\text{ for } k \geq 1, \\
        0, \quad &\text{ for } k=0.
    \end{cases} \quad \text{ and } \quad 
    \frac{dT_k}{dx} = \begin{cases} 
        k C^{(1)}_{k-1}, \quad &\text{ for } k \geq 1, \\
        0, \quad &\text{ for } k = 0,
    \end{cases}
\end{align}
such that $\lambda$-fold differentiation between Chebyshev and order-$\lambda$ 
ultraspherical polynomials is represented by the matrix
\begin{align}\nonumber
    \mathbf{D}_\lambda = 2^{\lambda-1} (\lambda-1)! 
\left (
\begin{array}{rrrrrrr}
\overmat{\lambda \text{times}}{0 & \dots & 0} &\lambda & & &  \\
        & & & & \lambda + 1 & &  \\
        & & & & & \lambda + 2 & \\
        & & & & & &  \ddots
\end{array}\right )
\end{align}

To compute the matrix representation $\tilde{\mathbf{L}}$ of the operator 
$\mathcal{L}$, coefficient functions $a_j$ are resolved to machine accuracy in 
terms of ultraspherical polynomials of order $j$. Due to a convolution formula 
similar to that for Chebyshev polynomials (\ref{eq:chebyshev-convolution}), 
multiplication by the expansion of $a_j$ can be written as a matrix $\mathbf{M
}_j(a_j)$ acting on the coefficients of an order-$j$ ultraspherical series, see 
equations (3.6) to (3.9) in \cite{olver_townsend}.
Coefficients of a $C^{(\lambda)}$-series are converted to a $C^{(\lambda+1)}
$-series by applying the matrix 
\begin{align}
    \label{eq:conversion-matrix-ultrasphericals}
    \mathbf{S}_\lambda = \begin{pmatrix}
        1 & & -\frac{\lambda}{\lambda + 2} & & & \\
        & \frac{\lambda}{\lambda+1} & & -\frac{\lambda}{\lambda+3} & & \\
        & & \frac{\lambda}{\lambda+2} & & -\frac{\lambda}{\lambda+4} & \\
        & & & \ddots & & \ddots
    \end{pmatrix},
\end{align}
to its vector of coefficients while a Chebyshev series can be converted to a 
$C^{(1)}$-series with the operator
\begin{align}
    \label{eq:conversion-matrix-chebyshev}
    \mathbf{S}_0 = \begin{pmatrix}
        1 & & -\frac{1}{2} & & & \\
        & \frac{1}{2} & & -\frac{1}{2} & & \\
        & & \frac{1}{2} & & -\frac{1}{2} & \\
        & & & \ddots & & \ddots
    \end{pmatrix}.
\end{align}

Combining these steps, the differential operator is found by computing
\begin{align}
\label{eq:preconditioned-matrix-spectral}
\tilde{\mathbf{L}} = \mathbf{D}_N + \mathbf{S}_{N-1} \mathbf{M}_{N-1}(a_{N-1}) 
\mathbf{D}_{N-1} + \dots + \mathbf{S}_{N-1}\dots \mathbf{S}_{0} \mathbf{M}_0(a_0)
\end{align}

The matrices $\mathbf{D}_j$ are composed of a single off-diagonal, while the matrices 
$\mathbf{S}_j$ constist of the main-diagonal and one off-diagonal. This results 
in a banded matrix $\tilde{\mathbf{L}}$.

As before, this matrix is truncated to size $(n-N) \times n$ and complemented with
the $N$ boundary conditions to obtain the system matrix $\tilde{\mathbf{A}}(n)$. 
Numerical experiments indicate that the condition number of these matrices grow as 
$\mathcal{O}(n)$ \cite[\S~3.3]{olver_townsend}. Applying the diagonal preconditioner
\begin{align}
    \mathbf{R} = \frac{1}{2^{j-1} (j-1)!} \, \text{diag} \left( \underbrace{1, 
    \dotsc, 1}_{N \text{ times}}, \frac{1}{N}, \frac{1}{N+1}, \dotsc \right) 
    \label{eq:olver-townsend-right}
\end{align}
from the right, it is shown in \cite[\S~4]{olver_townsend} that 
\begin{align}
    \label{eq:identity-plus-compact-spectral}
    \mathbf{A} \mathbf{R} = \mathbf{I} + \mathbf{K}_n,
\end{align}
where sequence of $n \times n$ matrices $\mathbf{K}_n$ converges to a compact 
operator $\mathcal{K} : \ell_\lambda^2 \rightarrow \ell_\lambda^2$ between the 
Hilbert spaces
\begin{align}\nonumber
    \ell_\lambda^2 = \left\{ \mathbf{u} = (u_k)_{k \in \naturals} : \Vert 
    \mathbf{u} \Vert_{\ell_\lambda^2} = \sqrt{\sum_{k=0}^\infty u_k^2 
    (1+k^2)^\lambda } \right\},
\end{align}
and the range of possible $\lambda$ determined by the boundary conditions 
$\mathcal{B}$. For Dirichlet boundary conditions, this includes $\lambda = 0$. 
By uniform convergence of orthogonal projection in the spectral bases, if 
$\mathcal{I} + \mathcal{K}$ is invertible, the condition number of the matrices 
$\mathbf{A} \mathbf{R}$ in the relevant $\ell_\lambda^2$-norm converges to that 
of $\mathcal{I} + \mathcal{K}$ \cite[\S~4]{olver_townsend}. In other words, 
growth of the condition number as $n \rightarrow \infty $ has been improved 
from $\mathcal{O}(n^{2N})$ to $\mathcal{O}(1)$ by a change of basis and the 
multiplication with a diagonal operator $\mathbf{R}$.

\subsection{Basis change as FOP}
\label{sec:generalized-integration-spectral}

While it is clear that the application of the operator $\mathcal{R}$ can be 
seen as a form of FOP, the FOP present in the basis change needs explicit 
formulation. We define the basis-change preconditioner
\begin{align}
    \label{eq:spectral-rl}
    \mathcal{R}_l(C^{(N)}_k) = T_k,
\end{align}
mapping an ultraspherical polynomial of order $N$ and degree $k$ to the 
Chebyshev polynomial of degree $k$. By extending $\mathcal{R}_l$ to linear 
combinations and then to series, the operator is well-defined on the space 
of $L^2(\rho^{(N)})$ with weight $\rho^{(N)}(x) = (1-x^2)^{N - 1/2}$.

Applying this operator from the left to the original equation $\mathcal{L} u 
= f$ and decomposing in terms of the Chebyshev test basis with the $\rho^{(0)}
$-scalar product leads to the same matrix as decomposing the original equation 
directly in terms of the ultraspherical test basis with the $\rho^{(N)}$-scalar 
product: Recall that the pure-Chebyshev method leads to the matrix
\begin{align}\nonumber
    \mathbf{L}[j,k] = \frac{(T_j, \mathcal{L} T_k)_{L^2(\rho^{(0)})}}{\Vert T_j 
    \Vert^2_{L^2(\rho^{(0)})}},
\end{align}
while the Chebyshev-ultraspherical method results in
\begin{align}\nonumber
\tilde{\mathbf{L}}[j,k] = \frac{(C^{(N)}_j, \mathcal{L} T_k)_{L^2(\rho^{(N)})}}{ 
\Vert C^{(N)}_j \Vert^2_{L^2(\rho^{(N)})}}.
\end{align}
Multiplying the original equation with $\mathcal{R}_l$ and using the pure-Chebyshev 
formula, we obtain
\begin{align}
    \label{eq:chebyshev-ultraspherical-change}
    \begin{split}
\frac{(T_j, \mathcal{R}_l \mathcal{L} T_k)_{L^2(\rho^{(0)})}}{ \Vert T_j \Vert^2_{
L^2(\rho^{(0)})} } 
& = \frac{(T_j, \mathcal{R}_l \sum_{m=0}^\infty C_m^{(N)} \tilde{\mathbf{L}}[m,k] 
)_{L^2(\rho^{(0)})} }{ \Vert T_j \Vert^2_{L^2(\rho^{(0)})} } \\
& = \sum_{m=0}^\infty \frac{(T_j, T_m \tilde{\mathbf{L}}[m,k])_{L^2(\rho^{(0)})} }
{ \Vert T_j \Vert^2_{L^2(\rho^{(0)})} }
= \tilde{\mathbf{L}}[j,k],
    \end{split}
\end{align}
the same matrix as in the Chebyhsev-ultraspherical case.

Hence, the basis change between Chebyshev and ultraspherical polynomials is a 
case of operator FOP. Together with the right-preconditioner
\begin{align}
    \label{eq:spectral-rr}
    \mathcal{R}_r(T_k) = \frac{1}{2^{N-1} (N-1)!} \begin{cases}
        T_k, \quad \text{ if } k < N, \\
        \frac{T_k}{k}, \quad \text{ if } k \geq N,
    \end{cases}
\end{align}
it serves to bound the operator $\mathcal{L}$ on the space $L^2(\rho^{(0)})$. 
Without preconditioning, the order-$N$ differential operator $\mathcal{L}$ is 
unbounded as an operator $\mathcal{L} : L^2(\rho^{(0)}) \rightarrow L^2(\rho^{(0)})$. 
After applying preconditioners from the left and the right, it is shown that 
the representation in terms of Chebyshev polynomials of the operator 
$\mathcal{R}_l \mathcal{L} \mathcal{R}_r : L^2(\rho^{(0)}) \rightarrow L^2(
\rho^{(0)})$ equals the identity plus a compact operator. By the isometry between 
$L^2(\rho^{(0)})$ and $\ell^2$, this implies the continuity of $\mathcal{R}_l 
\mathcal{L} \mathcal{R}_r$.

In fact, the two preconditioners serve as an $N$-fold integration, inverting the 
$N$-th derivative
\begin{align}\nonumber
    \mathcal{R}_l \frac{d^N}{dx^N} \mathcal{R}_r = \mathcal{I} : L^2(\rho^{(0)}) 
    \rightarrow L^2(\rho^{(0)}).
\end{align}
In this, aside from providing the factor $1/(2^{N-1} (N-1)!)$, the right 
preconditioner serves to counteract an assymetry in the definition of Chebyshev 
and ultraspherical polynomials. While the differentiation of an ultraspherical 
polynomial $C^{(\lambda)}_k$ does not lead to a factor depending on the degree 
$k$ of the polynomial, differentiating the Chebyshev polynomial $T_k$ does. This 
single linear scaling coming from the first change from $T_k$ to $C^{(1)}_k$-series 
is cancelled by $\mathcal{R}_r$.

In infinite-precision arithmetic, traditional matrix preconditioning with the 
truncated diagonal operator $\mathcal{R}$ from the right and the matrix version 
of $\mathcal{R}_l$ with entries
\begin{align}\nonumber
    \mathbf{R}_l[j,k] = (T_j, \mathcal{R}_l(T_k))_{L^2(\rho^{(0)})}
\end{align}
from the left would lead to a solution equivalent to that of the 
Chebyshev-ultraspheri-cal method. Given that precision is infinite, numerical error 
does not play a role, and a speedup of iterative methods would occur as predicted by 
the condition-number improvement induced by the preconditioning. 
In finite-precision arithmetic, this suffers from the multiplication of the 
ill-conditioned matrices $\mathbf{L}$ and $\mathbf{R}_l$. 

For FOP it is thus \emph{crucial} that the matrix $\tilde{\mathbf{L}}$ is computed 
to machine precision by use of recursion relations as in equation 
(\ref{eq:preconditioned-matrix-spectral}) instead of by evaluating the matrix 
product $\tilde{\mathbf{L}} = \mathbf{R}_l \mathbf{L}$.

Similarly, the right-hand side $f$ has to be discretized directly in terms of 
ultraspherical polynomials. The alternative, a conversion of the Chebyshev 
discretization $\mathbf{f}$ with components $(T_j, f)$ into the ultraspherical 
representation by forming the product with $\mathbf{R}_l$, suffers again from 
the bad conditioning of the matrix multiplication.

\section{FOP for finite-element methods} \label{chap:fem}

As laid out in the previous section, $N$-fold integration is a potential 
preconditioner for normalized differential operators of order $N$. In the context 
of spectral methods, we relied on recursive relationships between Chebyshev and 
ultraspherical polynomials to construct the algorithm. Now, we present an 
application of integration preconditioning for finite-element methods.

Here, we focus on fourth-order differential equations in one dimension with 
Dirichlet and Neumann boundary conditions, approximating functions by the 
\emph{cubic Hermite element} \cite[\S~3.2]{brenner2007mathematical}. 

In this section, we briefly give an overview of the treatment of fourth-order 
differential equations with the finite element method and then describe the 
algorithm used for FOP. We show that our new method successfully improves the 
accuracy of solutions to the biharmonic equation and other fourth-order linear 
differential equations by avoiding an otherwise catastrophic increase of the 
condition number.

\subsection{Finite elements for fourth-order differential equations} 
\label{sec:fem-unpreconditioned}

We consider linear, fourth-order differential equations of the form
\begin{align}\label{eq:pdefourth}
\begin{cases}
\mathcal{L} u = \diff{4} u + a_3 \diff{3} u + a_2 \diff{2} u 
+ a_1 \diff{1} u + a_0 u = f, \quad \text{ on } (-1, 1)\\
u(\pm 1) = u'(\pm 1) = 0,
\end{cases}
\end{align}
where $a_i: (-1, 1) \rightarrow \reals$, $i=0, \dotsc, 4$ are smooth functions.

Let $H^1(-1,1)$, $H^1_0(-1,1)$ and $H^2(-1,1)$ be Sobolev spaces defined as 
usual~\cite[Ch.~1]{brenner2007mathematical}, and $L^2(-1,1)$ be the space of 
square-integrable functions in $(-1,1)$. In addition, we introduce the 
space $H^2_0(-1,1):= \lbrace u \in H^2(-1,1) \, : \,: u(\pm 1) = u'(\pm 1) = 0
\rbrace$.

The \emph{weak formulation} corresponding to \eqref{eq:pdefourth} is: find $u$ 
such that
\begin{align}\nonumber
    \mathsf{a}(u, w) = (f, w)_{L^2(-1,1)}, \qquad \forall w \in H^2_0(-1,1),
\end{align}
where we introduced the bilinear form $\mathsf{a} : H^2_0(-1,1) \times H^2_0(-1,1) 
\rightarrow \reals $ defined as
\begin{align}\nonumber
\mathsf{a}(w, u) &= (w, \mathcal{L} u)_{L^2(-1,1)} \nonumber \\
&= (w'' - (a_3 w)', u'')_{L^2(-1,1)} + (-(a_2 w)' + a_1 w, u')_{L^2(-1,1)} 
\nonumber\\ &\quad+ (a_0 w, u)_{L^2(-1,1)},
\end{align}
For simplicity, we assume that the coefficient functions $a_j$ are such 
that the bilinear form $\mathsf{a}$ is continuous and elliptic in $H^2_0(-1,1)$.

We choose \emph{Hermite finite elements} for the trial and test basis 
\cite[Chapter 10]{lecture_notes_fem} on a uniform mesh for $(-1,1)$. As before, 
we turn the differential equation into a $2n \times 2n$ linear system $\mathbf{L} 
\mathbf{u} = \mathbf{b}$ with
\begin{align}\nonumber
\mathbf{L}[j,k] = \mathsf{a}(\phi_j, \phi_k)_{L^2(-1,1)}, \quad \text{ for } j, 
k \in \{1, \dotsc, 2n \}
\end{align}
and right-hand side
\begin{align}\nonumber
   \mathbf{b}[j] = (\phi_j, f)_{L^2(-1,1)}, \quad \text{ for } j \in \{1, \dotsc, 2n \}.
\end{align}

For the biharmonic equation with $a_j = 0$ for $j=0, \dotsc, 3$, the matrix 
$\mathbf{L}$ has condition number observed to be increasing proportionally to $n^4$. 
At the same time, the use of Hermite elements leads to fast convergence of the error: 
Figure \ref{fig:fem-biharmonic-nonpreconditioned} shows the relative error of the finite 
element approximation to the true solution $u = (1-x^2)^2$ measured in the $H^2$-norm,
\begin{align}\nonumber
    \Vert u \Vert_{H^2(-1,1)}^2 = \int_{-1}^1 \left( u^2 + (u')^2 + (u'')^2 \right) \dint x
\end{align}
and $L^2(-1,1)$ norm, respectively, as well as a multiple of the condition number. 
For $n < 1200$, the error decreases as $n^{-2}$. This is the expected discretization 
error in $H^2$-norm for Hermite elements \cite[Chapter 2.4]{stoer_bulirsch_book}. 
If $n$ is increased further, the error starts behaving erratically and begins to 
increase approximately proportional to $\kappa(\mathbf{L})$, with the constant 
of proportionality close to machine accuracy. This indicates that discretization 
error is being overtaken by numerical error beyond that point.

In the $L^2(-1,1)$-norm, Hermite elements guarantee $\mathcal{O}(n^{-4})$ convergence 
of the discretization error for smooth solutions \cite[Chapter 2.4]{stoer_bulirsch_book}. 
In the present case, this means that numerical issues overtake discretization as the 
main error source already at $n \approx 300$, which is also depicted in Figure 
\ref{fig:fem-biharmonic-nonpreconditioned} together with the comparison to $n^{-4}$.

\begin{figure}[h!] \centering
\includegraphics[width=0.8\textwidth]{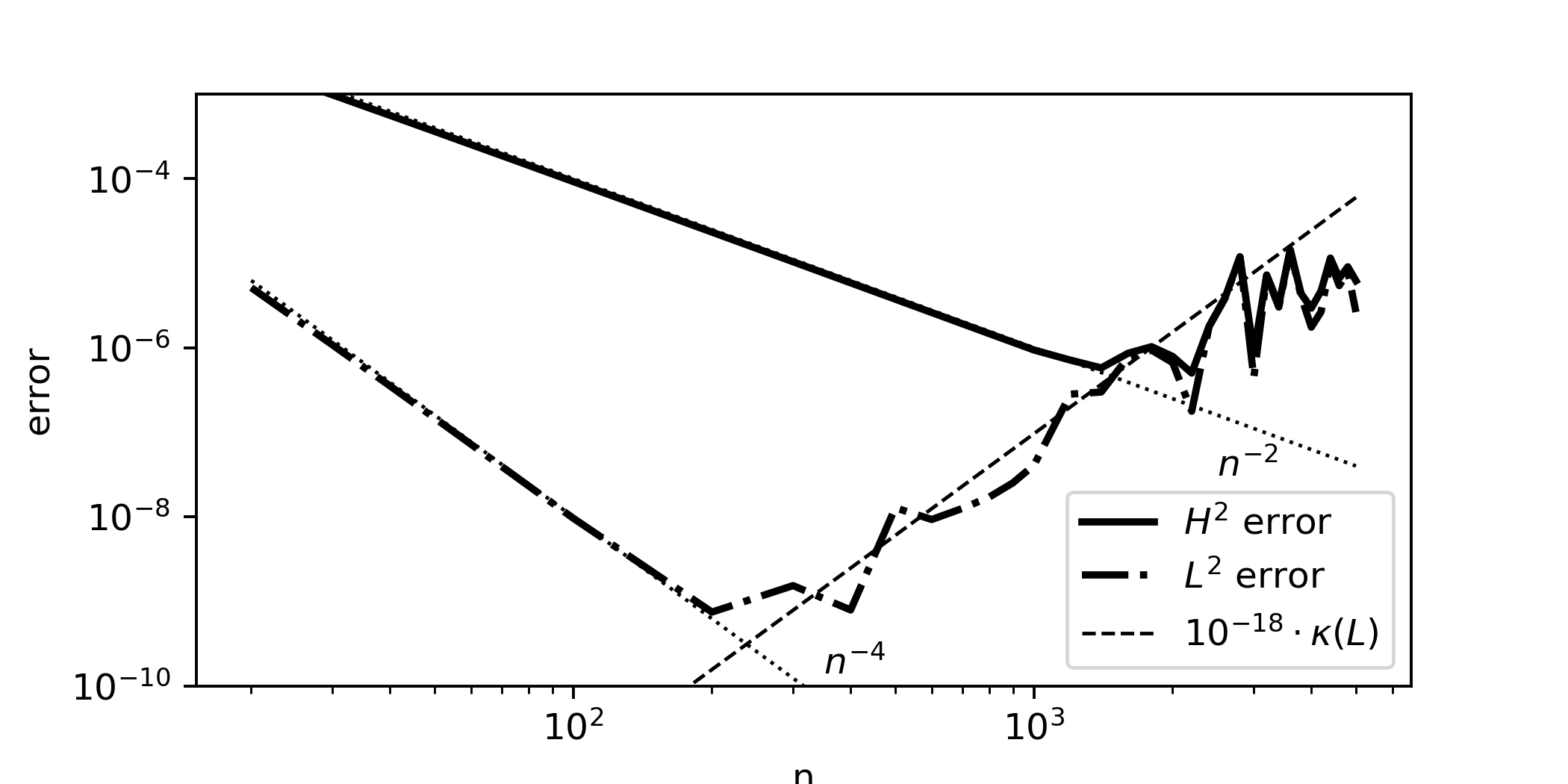}
\caption{Relative error $\Vert u - \hat{u} \Vert / \Vert u \Vert$ as a function 
of the number of cells $n$ of the finite-element approximation $\hat{u}$ to the 
true solution $u = (1-x^2)^2$ of the biharmonic equation with constant right-hand 
side $u^{(4)}(x) = 24$.  The solution is approximated with Hermite elements, and 
error is computed in $H^2$ and $L^2$ norms. Integrals were computed using 
Gaussian quadrature with $11$ quadrature points per cell. Also shown are the 
expected $\mathcal{O}(n^{-2})$-scaling of the error and the multiple $\epsilon 
\kappa(\mathbf{L})$ of the involved matrix, where $\epsilon = 10^{-18}$ is close 
to machine precision.}
\label{fig:fem-biharmonic-nonpreconditioned}
\end{figure}

\subsection{Integration preconditioning for fourth-order differential equations}
\label{sec:fem-integration}

Consider now an operator $\mathcal{R}$ to be used as a right preconditioner for 
the differential equation $\mathcal{L} u = f$. Replacing $u$ by $\mathcal{R}v $, 
where $v$ is any function that is mapped by $\mathcal{R}$ to the same smoothness 
and boundary properties as $u$, 
\begin{align}
    \mathcal{R} v \in H_0^2(-1,1), \label{eq:label-conforming-condition} 
\end{align}
we find for $w \in H^2_0(-1,1)$ that $\mathcal{R}v $ may also replace $u$ in the 
bilinear form $\mathsf{a}(w, \mathcal{R} v)$.

The system matrix for the preconditioned operator equation is again obtained by 
computing the bilinear form on all $2n \times 2n$ pairs of Hermite basis 
functions and is given by its entries
\begin{align}\nonumber
    \tilde{\mathbf{L}}[j,k] = \mathsf{a}(\phi_j, \mathcal{R} \phi_k).
\end{align}
Notably, the knowledge of $\mathsf{a}( \cdot , \mathcal{R} \cdot)$ for arbitrary 
arguments is not required. Instead, for the computation of the matrix evaluation 
of the preconditioner on the elements $\phi_k$ of the Hermite basis and 
computation of $\mathsf{a}$ on pairs of $\phi_j $ and $\mathcal{R} \phi_k$ is 
sufficient.

For a fourth-order differential operator with leading-order term $\frac{d^4}{dx^4}$, 
we use four-fold integration as the preconditioner. Moreover, for $\mathcal{L}: 
H^2_0(-1,1) \to H^{-2}(-1,1) $, the four-fold integration preconditioner is 
chosen such that it takes care of boundary conditions, i.e. $\mathcal{R}: H^2_0
(-1,1) \to H^{-2}(-1,1) $.

Computing $\Phi_k = \mathcal{R} \phi_k$ for $k \in \{1, \dotsc, 2n \}$ is 
straightforward. During the cell-wise integration of $\phi_k$, $k \in \{1, 
\dotsc, 2n \}$, in each of the $n+1$ cells, $4$ degrees of freedom arise in the 
form of integration constants. The first $4n$ of these are chosen such that 
$\Phi_k$ and its first three derivatives are continuous on the boundaries between 
the cells. Because $\phi_k$ is continuous with continuous first derivative, it 
follows that $\Phi_k \in H^2(-1,1)$. The last four integration constants are 
determined by the boundary conditions of $H^2_0(-1,1)$. Hence, $\Phi_k \in H^2_0
(-1,1)$ is guaranteed.

The resulting functions $\Phi_k = \mathcal{R} \phi_k$ are used as the trial-space 
basis. The matrix $\tilde{\mathbf{L}}$ is computed numerically using Gaussian 
quadrature. After solving the system $\tilde{\mathbf{L}} \mathbf{v} = \mathbf{b}$, 
we reconstruct the solution $\hat{u} = \sum_{k=1}^{2n} v_k \Phi_k$.

\begin{figure}
    \centering
    \includegraphics[width=0.8\textwidth]{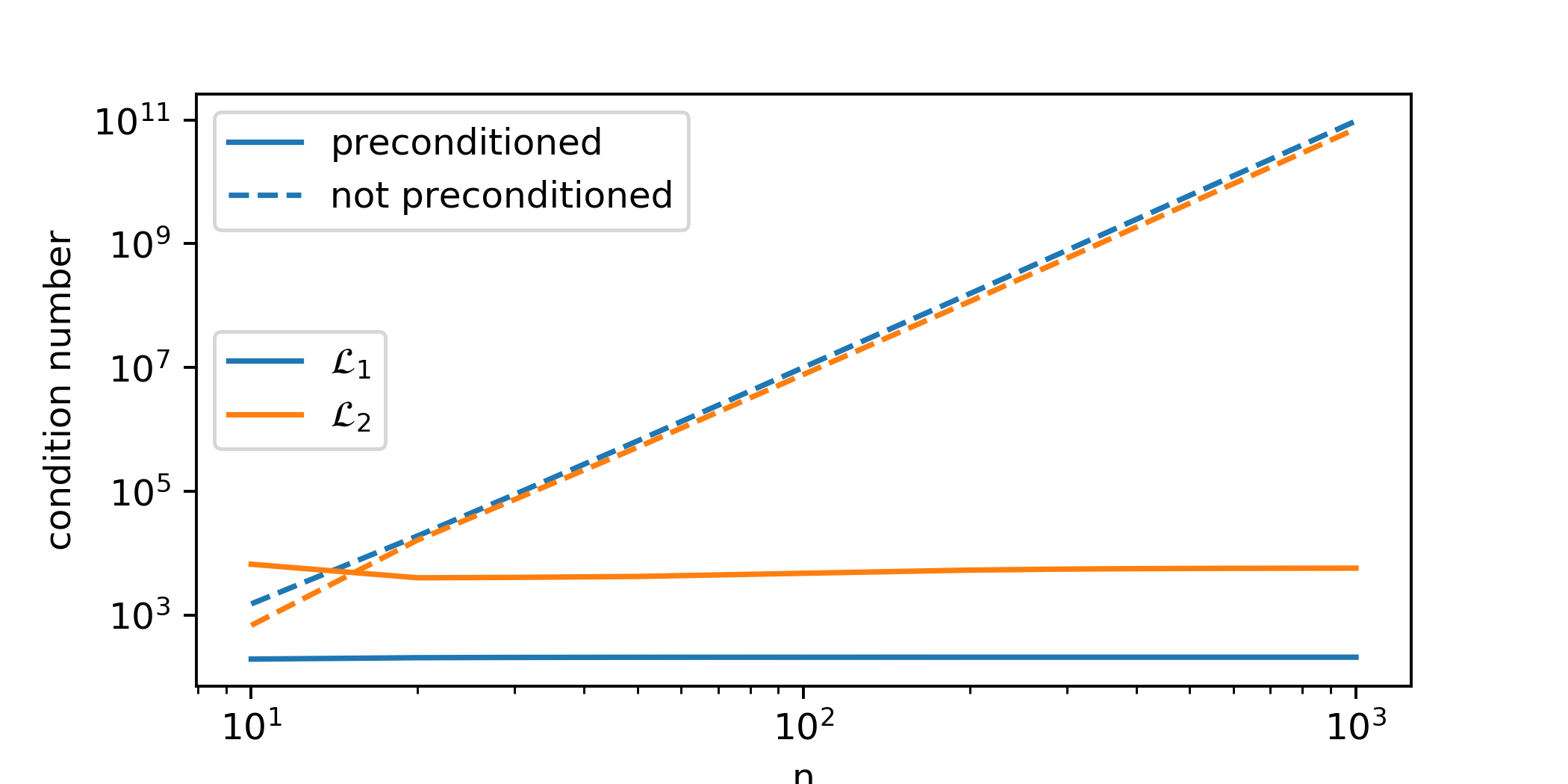}
    \caption{Condition numbers of the non-preconditioned and preconditioned 
    method for the biharmonic equation $\mathcal{L}_1$ and the differential 
    operator $\mathcal{L}_2$ defined in equation (\ref{eq:fem-test-dop-example})}
    \label{fig:fem-comparison-condition-numbers}
\end{figure}

Instead of growing as $\mathcal{O}(n^4)$, the condition number of $\tilde{\mathbf{L}}$ 
approaches a constant as $n \to \infty$. Figure \ref{fig:fem-comparison-condition-numbers} 
shows the condition number of the matrices $\mathbf{L}$ and $\tilde{\mathbf{L}}$ for 
the biharmonic equation
\begin{align}
    \mathcal{L}_1 = \diff{4} u \label{eq:fem-biharmonic-operator}
\end{align}
and for the equation with operator 
\begin{align}
    \mathcal{L}_2 = \diff{4} + \alpha \sin(20 \pi x) \diff{3} + \alpha \cos(20 \pi x^3) 
    \diff{2} + \frac{\alpha x}{1 + x^2}
    \label{eq:fem-test-dop-example}
\end{align}
with $\alpha = 200$. We see that up to $n = 1000$, the condition number of the 
preconditioned matrices does not increase beyond $~100$ and $10000$, respectively, 
whereas condition numbers obtained from the standard algorithm show a clear 
$\mathcal{O}(n^4)$ growth. 

\begin{figure} \centering
\includegraphics[width=\textwidth]{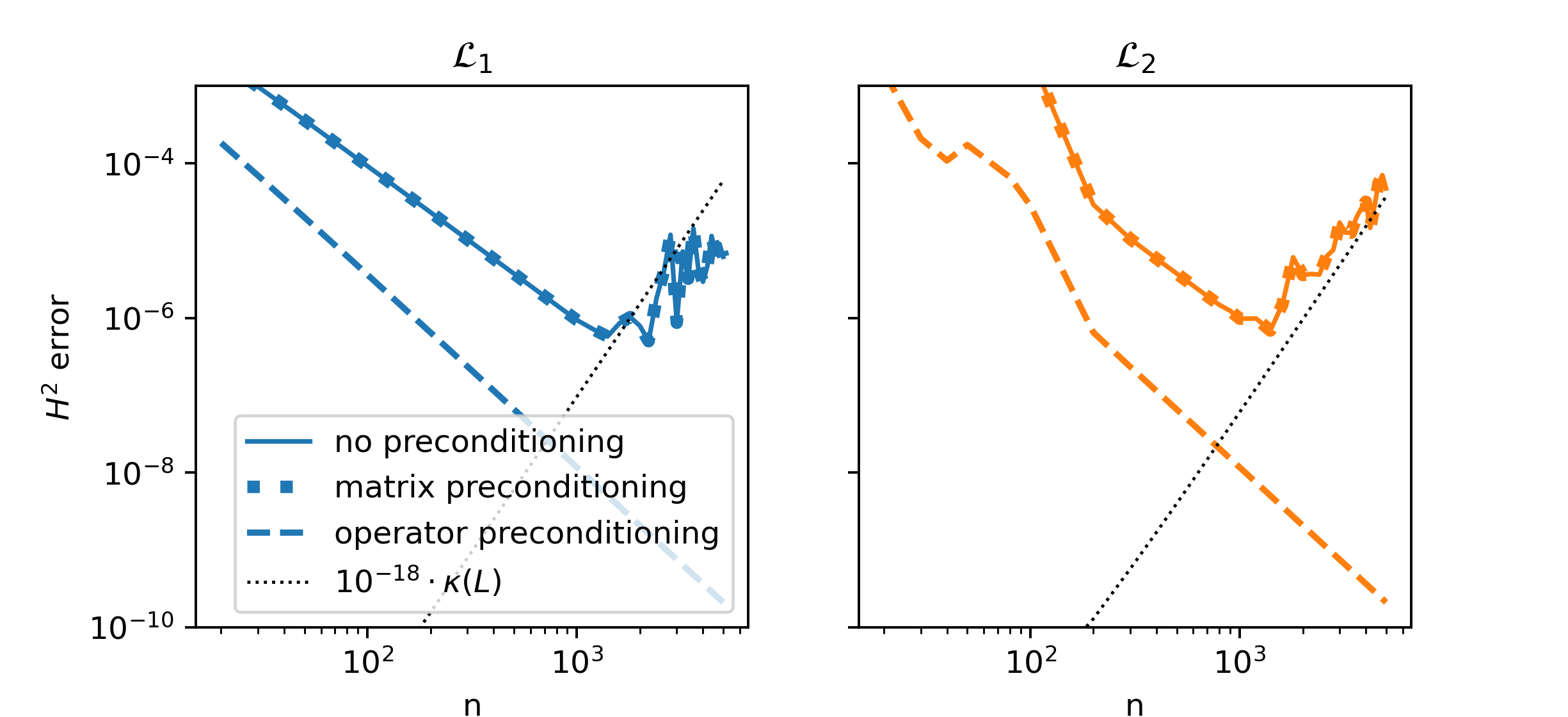}
\caption{Relative error of the approximation to the solution $u = (1-x^2)^2$ 
obtained with non-preconditioned, matrix-preconditioned with the inverse of the 
stiffness matrix of the biharmonic equation, and operator-preconditioned Hermite 
FEM for the two differential equations with differential operators $\mathcal{L}_1$ 
and $\mathcal{L}_2$ defined in equations (\ref{eq:fem-biharmonic-operator}) and 
(\ref{eq:fem-test-dop-example}). Note that matrix preconditioning and the 
unpreconditioned method lead to the same relative error.}
\label{fig:fem-comparison-errors}
\end{figure}

Accordingly, solutions obtained from the preconditioned method do not suffer 
from an erratic increase of error for high $n$ as in Figure 
\ref{fig:fem-biharmonic-nonpreconditioned}. A comparison is shown in Figure 
\ref{fig:fem-comparison-errors}. As before, for the nonpreconditioned method, 
the error starts to increase once $\epsilon \kappa(\mathbf{L})$ grows to the 
range of the discretization error and it remains above $10^{-7}$ for all $n$. 
The same holds for matrix preconditioning: The matrix equivalent of four-fold 
differentiation is the finite-element discretization of the biharmonic operator 
$\frac{d^4}{dx^4}$. Using the inverse of this matrix as a preconditioner leads 
to no visible changes in the relative error. 
On the other hand, FOP is able to find solutions to the problem with 
error as low as $10^{-10}$ at $n=5000$, and with no deviations from the trend 
of decreasing error when $n$ is increased. 

This accuracy is possible because the condition number of the FOP system matrix 
$\tilde{\mathbf{L}}$ remains bounded. 
As we shall see,
this is a consequence of two conditions: 

$i)$ the operator $\tilde{\mathcal{L}} := \mathcal{L} \mathcal{R}$ is an 
\emph{endomorphism} in $H^{-2}(-1,1)$ \cite{CHR02,hiptmair}. 

Indeed, this 
condition allows us to state the following: 
Let $\mathsf{\tilde{b}}(\cdot, \cdot) : H^{-2}(-1,1) \times H^{-2}(-1,1) 
\to \mathbb{R}$ be the bilinear form of $\tilde{\mathcal{L}}$. By the properties 
of $\mathcal{L}$ and $\mathcal{R}$, 
there exist constants $ \tilde{\beta},\tilde{C}_b\in\mathbb{R}>0$ such that
\begin{align}\label{eq:condFop1}
 \tilde{\beta} \, \Vert u \Vert_{H^{-2}(-1,1)}^2 \leq |\mathsf{\tilde{b}}(u, u) 
 |\leq \tilde{C}_b \, \Vert u \Vert_{H^{-2}(-1,1)}^2 \qquad \forall u \in H^{-2}(-1,1).
 \end{align}

It is worth noticing that using $\mathcal{R}$ as a left preconditioner would 
also lead to a suitable FOP. Indeed, in that case we would have the endomorphism 
$\widehat{\mathcal{L}}:=\mathcal{R} \mathcal{L}: H^{2}_0(-1,1) \to H^{2}_0(-1,1)$, 
satisfying for some $\widehat{\beta},\widehat{C}_b>0$
\begin{align}\label{eq:condFop2}
 \widehat{\beta} \, \Vert u \Vert_{H^{2}_0(-1,1)}^2 \leq |\mathsf{\widehat{b}}(u, u) 
| \leq \widehat{C}_b \, \Vert u \Vert_{H^{2}_0(-1,1)}^2 \qquad \forall u \in H^{2}_0(-1,1).
 \end{align}

$ii)$ The Hermite basis functions are \emph{well-conditioned}.

In order to see this, we examine the mass matrix 
\begin{align}
    \label{eq:definition-mass-matrix}
    \mathbf{M}[j,k] = (\phi_j, \phi_k)_{L^2(-1,1)}. 
\end{align} 

Recalling the definition and support of the Hermite functions~\cite[\S~3.2]{brenner2007mathematical}, 
integration yields the structure of the mass matrix
\begin{align}\label{eq:M}
 \mathbf{M} = \Delta x \begin{pmatrix} \mathbf{A} & \mathbf{B} & & \\
                     \mathbf{B}^{\top} & \mathbf{A} & \mathbf{B} & \\
                     & \mathbf{B}^{\top} & \mathbf{A} & \mathbf{B} \\
                     && \ddots & \ddots 
                     \end{pmatrix} \in \reals^{2n \times 2n}
\end{align}
with the $2 \times 2$ blocks
\begin{align}\nonumber
\mathbf{A} &= \begin{pmatrix} 26/35 & 0 \\ 0 & 2/105 \end{pmatrix},
& \qquad \mathbf{B}= \begin{pmatrix} 9/70 & 13/420 \\ -13/420 & -1/140 \end{pmatrix}.
\end{align}

\begin{proposition}\label{prop:condM}
For all $n \in \naturals$, the matrix $\mathbf{M}$ in~\eqref{eq:M} satisfies
\begin{align}\nonumber
\kappa(\mathbf{M})\leq 39\frac{1+\delta}{1-\delta}, \qquad \delta 
= \frac{3+\sqrt{13/3}}{8}\ (<0.64).
\end{align}
\end{proposition}
\begin{proof}
The idea is to use Gerschgorin's theorem to a diagonally scaled matrix 
$\tilde{\mathbf{M}}:=\mathbf{D}\mathbf{M}\mathbf{D}$, where 
$\mathbf{D}=\mbox{diag}(\mathbf{D}_1,\mathbf{D}_1,\ldots)\in\reals^{2n \times 
2n}$, with $\mathbf{D}_1 = \big(
\begin{smallmatrix}
 \sqrt{35/26} & 0 \\ 0 & \sqrt{105/2}  
\end{smallmatrix}
\big)$; that is, we apply a diagonal scaling such that 
$\mathbf{D}\mathbf{M}\mathbf{D}$ has 1's on the diagonal\footnote{This is 
equivalent to normalizing the basis functions to have unit norms; which is 
known to minimize the condition number up to a factor $\sqrt{n}$ with a 
diagonal scaling~\cite[Thm.~7.5]{higham}.}. 

Then $\tilde{\mathbf{M}}$ is in the same form as $ \mathbf{M}$~\eqref{eq:M}, 
with $\mathbf{A},\mathbf{B}$ replaced with $\tilde{\mathbf{A}},\tilde{\mathbf{B}}$ 
respectively, where 
\begin{align}\nonumber
\tilde{\mathbf{A}} &= \begin{pmatrix} 1 & 0 \\ 0 & 1 \end{pmatrix},
& \qquad 
\tilde{\mathbf{B}}= \begin{pmatrix} 9/52 & \sqrt{13}/(8\sqrt{3}) \\ -\sqrt{13}/
(8\sqrt{3}) & -3/8 \end{pmatrix}.  
\end{align}
Now by Gerschgorin's theorem, (and since $\tilde{\mathbf{M}}$ is symmetric) 
the eigenvalues of $\tilde{\mathbf{M}}$ must lie in $[1-\frac{1}{8}(3+\sqrt{13/3}),
1+\frac{1}{8}(3+\sqrt{13/3})]=[1-\delta,1+\delta]$. Since this is a positive 
interval, it follows that $\tilde{\mathbf{M}}$ is positive definite and 
the eigenvalues are equal to the singular values. Therefore 
$\kappa(\tilde{\mathbf{M}})\leq \frac{1+\delta}{1-\delta}$. 
Finally, $\kappa(\mathbf{M})\leq \kappa(\tilde{\mathbf{M}})\kappa(\mathbf{D})^2$
and $\kappa(\mathbf{D})^2=\kappa(\mathbf{D}_1)^2=39$, completing the proof. 
\end{proof}

Next, we proceed to prove in a more general way why these two conditions 
guarantee that we arrive at a well-conditioned system.
\begin{theorem}\label{thm:Cond}
Let $X$ be a Hilbert space, and $\mathsf{b}$ a bounded bilinear form $\mathsf{b}(\cdot, 
\cdot) : X \times X \to \mathbb{R}$ 
so that there exists $M>0$ such that $|\mathsf{{b}}(v, u)| \leq M \Vert u \Vert_{X}
\Vert v \Vert_{X}$ for all $v \in X$. 
Suppose that $\mathsf{b}$ further satisfies
\begin{align}\label{eq:bound}
 {\beta} \, \Vert u \Vert_{X}^2 \leq |\mathsf{{b}}(u, u) |
 \qquad \forall v \in X,
 \end{align}
for some $\beta>0$. 

Let $X_h \subset X$ be a finite dimensional space such that $\mathsf{dim}(X_h) = N$ and 
$X_h = \mathsf{span} \lbrace q_i \rbrace_{i=1}^N$. 
Then the matrix $\mathbf{B}_h$ defined by 
$(\mathbf{B}_h)_{ij} = \mathsf{b}(q_j, q_i)$ satisfies
\begin{align}\label{eq:Bhprops}
\kappa(\mathbf{B}_h) \leq \frac{M}{\beta}(\kappa(\mathcal{Q}))^2,
\end{align}
where $\mathcal{Q}=[q_1,\ldots,q_N]$ is a quasimatrix\footnote{A quasimatrix has 
(among other decompositions inheriting matrix decompositions) the QR 
factorization~\cite{trefethen2010householder} $\mathcal{Q}=\mathcal{\tilde Q}R$, 
where $\mathcal{\tilde Q}$ has orthonormal columns (with respect to $(\cdot,\cdot)_X$), 
and $R\in\mathbb{R}^{N\times N}$ is upper triangular. The condition number is 
defined by the matrix condition number $\kappa(\mathcal{Q})=\kappa(R)$.}, 
i.e., a matrix whose columns are functions (e.g.~\cite{chebfunofficial}). 
\end{theorem}
\begin{proof}
 First, let $B \in L(X,X)$ be the bounded linear mapping corresponding to the bilinear 
 form $\mathsf{b}$. 

For any $\mathbf{u}=(u_1,\ldots,u_N)^\top,\mathbf{v}=(v_1,\ldots,v_N)^\top \in
\mathbb{C}^N$ of unit norm $\|\mathbf{u}\|_2=\|\\mathbf{v}\|_2=1$, we have 
\[(\mathbf{B}_h\mathbf{v})[j] = \mathsf{b}(q_j, \sum_{\ell=1}^{N} v_\ell q_\ell) , 
\qquad 
\mathbf{u}^\top\mathbf{B}_h\mathbf{v} = \mathsf{b}(\sum_{\ell=1}^{N} u_\ell q_j, 
\sum_{\ell=1}^{N} v_\ell q_\ell) . 
\]
Hence by the assumption~\eqref{eq:bound} we have 
\begin{equation}  \label{eq:fov}
\beta\|\sum_{\ell=1}^{N} v_\ell q_\ell\|_X^2
 \leq |\mathbf{v}^\top\mathbf{B}_h\mathbf{v}| . 
\end{equation}
We can bound $\|\sum_{\ell=1}^{N} v_\ell q_\ell\|_X$ as
\begin{equation}  \label{eq:boundsqj}
\sigma_{\min} (\mathcal{Q})\leq 
\|\sum_{\ell=1}^{N} v_\ell q_\ell\|_X\leq \sigma_{\max}(\mathcal{Q})  .
\end{equation}
It follows that
$|\mathbf{v}^\top\mathbf{B}_h\mathbf{v}| \geq \sigma_{\min}^2 (\mathcal{Q})
\beta$ for any unit vector $\mathbf{v}$; this implies 
$\|\mathbf{B}_h\mathbf{v}\|_2 \geq \sigma_{\min}^2 (\mathcal{Q})\beta$ for 
any unit norm vector $\mathbf{v}$, and therefore $\sigma_{\min}(\mathbf{B}_h)
\geq \sigma_{\min}^2 (\mathcal{Q})\beta$. 

We next bound $\sigma_{\max}(\mathbf{B}_h)=\|\mathbf{B}_h\|_2$ from above. 
The first bound in~\eqref{eq:bound} and~\eqref{eq:Bhprops} yield 
$|\mathbf{u}^\top\mathbf{B}_h\mathbf{v}| \leq 
M\sigma_{\max}^2(\mathcal{Q})$, for any $\mathbf{u},\mathbf{v}$ of unit norm. 
This means $\mathbf{B}_h\leq M\sigma_{\max}^2(\mathcal{Q})$. 

Putting these together, 
we conclude that 
\[
\kappa(\mathbf{B}_h)\leq \frac{M\sigma_{\max}^2(\mathcal{Q})}{\sigma_{\min}^2 
(\mathcal{Q})\beta}=\frac{M}{ \beta}(\kappa(\mathcal{Q}))^2.
\]
\end{proof}

This result shows that the linear system is well-conditioned if the operator after 
FOP has a tightly bounded bilinear form, and a well-conditioned basis $\mathcal{Q}$ 
is used for the discretization, ideally $\kappa(\mathcal{Q})$ not growing with the 
discretization $N$. 

It is worth noting that the presence of $(\kappa(\mathcal{Q}))^2$ in~\eqref{eq:Lhkappa} 
appears to be necessary, and is not an artifact of the analysis. To see this, consider 
the case $\mathcal{L} \mathcal{R}=\mathcal{I}$ (the 'ideal FOP'); then $\tilde{\mathbf{L}}_h$ 
is the Gram matrix of $\mathcal{Q}$, so $\kappa(\tilde{\mathbf{L}}_h)=(\kappa(\mathcal{Q}))^2$. 

Let us now return to the specific examples in Figure~\ref{fig:fem-comparison-errors}.
Since the mass matrix \eqref{eq:definition-mass-matrix} has the property 
$\kappa(\mathbf{M})=(\kappa(\mathcal{Q}))^2$. Theorem~\ref{thm:Cond} and 
Proposition~\ref{prop:condM} imply that 
$\kappa(\tilde{\mathbf{L}}_{1h})$ is bounded independently of $n$ 
for~\eqref{eq:fem-biharmonic-operator}, for which $\beta=M=1$.

\subsection{Perturbed identity after FOP}
In many cases, such as \eqref{eq:fem-test-dop-example}, the operator after FOP is a 
perturbed identity $\mathcal{I} + \mathcal{K}$. In such cases we have the following.

\begin{corollary}\label{thm:IKrev}
Let $\mathcal{R}$ be such that 
\begin{equation}  \label{eq:LRpertI}
\mathcal{L} \mathcal{R} = \mathcal{I} + \mathcal{K} : X \to X,  
\end{equation}
where $\mathcal{I}$ is the identity operator, and 
$\mathcal{K}$ is bounded, i.e., 
 there exists $M_{\mathcal{K}}>0$ such that 
$(u,\mathcal{K}u)_X \leq M_{\mathcal{K}} \Vert u \Vert_{X}\Vert v \Vert_{X}$ for all $u,v\in X$. 
Suppose that
\begin{equation}
  \label{eq:alpbet}
\beta_{\mathcal{K}}\|u\|_X^2   \leq (u,\mathcal{K}u)_X \qquad \forall u\in X,
\end{equation}
 for some constant $\beta_{\mathcal{K}}>-1$.
Then we have 
\begin{equation}
\label{eq:Lhkappa}  
 \kappa(\tilde{\mathbf{L}}_h) \leq \frac{1+M_{\mathcal{K}}}{1+\beta_{\mathcal{K}}}(\kappa(\mathcal{Q}))^2
\end{equation}
where $\tilde{\mathbf{L}}_h$ is the Galerkin matrix of $\tilde{\mathcal{L}}
:=\mathcal{L} \mathcal{R}$ using the basis functions $\lbrace q_i \rbrace_{i=1}^N$,
and $\mathcal{Q}=[q_1,\ldots,q_N]$.
\end{corollary}
\begin{proof}
By the assumptions, we have 
\begin{align}
(1+\beta_{\mathcal{K}})\|u\|_X^2 &\leq (u, (\mathcal{I} + \mathcal{K}) u)_{X}, 
& (v, (\mathcal{I} + \mathcal{K}) u)_{X} \leq (1+M_{\mathcal{K}})\|u\|_X\|v\|_X
\end{align}
for all $u,v \in X$.
Therefore, the result follows from Theorem~\ref{thm:Cond}.
\end{proof}

\begin{remark}
Theorem~\ref{thm:IKrev} 
indicates that two conditions ensure $\tilde{\mathbf{L}}_h$ is well-conditioned: 
(i) that the FOP is effective so that the FOP'd operator is a ``small'' 
perturbation of identity, and (ii) a well-conditioned basis $\mathcal{Q}$ is 
chosen. We suspect that the assumptions in Corollary~\ref{thm:IKrev} are stronger 
than necessary, and that $\kappa(\tilde{\mathbf{L}}_h)$ could be bounded by a 
constant under a looser condition than~\eqref{eq:alpbet}.  
\end{remark}

\begin{remark}
A good FOP often results in the form~\eqref{eq:LRpertI} with a compact $\mathcal{
K}$. For instance, in the example \eqref{eq:fem-test-dop-example}, under the 
assumption of continuously differentiable coefficient functions and by compactness 
of the embeddings $H^{\lambda+1}(-1,1) \subset H^\lambda(-1,1)$ for $\lambda 
\geq 0$ \cite[Chapter 1]{grisvard}, we have $  \mathcal{L} \mathcal{R} = 
\mathcal{I} + \mathcal{K},$ where
\begin{align} \nonumber
    \mathcal{K} = \left( a_3 \diff{3} + a_2 \diff{2} + a_1 \diff{1} + a_0 \right) 
    \mathcal{R}
\end{align}
is a compact operator. 

The ``identity plus compact operator'' resembles the situation in Section 
\ref{sec:olver_townsend}. Compactness of $\mathcal{K}$ can be useful for 
verifying its properties, such as \eqref{eq:alpbet}. 
\end{remark}

Furthermore, the extension to $\mathcal{L} \mathcal{R} = \mathcal{I} + \mathcal{K}$ 
highlights another specialty of FOP: When investigating the properties of any 
form of preconditioning, we aim for statements about the resulting matrices. 
Analysis of matrix-preconditioning schemes may take place on the matrix level, 
for example by bounding the generalized Rayleigh coefficient
$(\mathbf{u}^{\top} \mathbf{L} \mathbf{u})/ (\mathbf{u} \mathbf{R}^{-1} 
\mathbf{u})$ \cite{wathen}. This is also possible when all elements and operators 
on the continuous space have representations in terms of infinite-dimensional 
matrices, such as in the case of spectral methods. For other cases of FOP, 
however, another option is to perform analysis on the levels of the operators 
themselves. This diffuses the classical distinction between numerical linear 
algebra and analysis of differential equations and calls for joint treatment 
of the whole solution process, a claim that is already being pushed for by other 
authors such as \cite{malek_strakos} and \cite{saad_book}.

\section{Discussion}

FOP can dramatically improve the accuracy of computed solutions in a variety of 
contexts. Our analysis identifies three properties required for a successful FOP. 
First, one needs to identify an operator such that its composition with $\mathcal{
L}$ is an endomorphism. Second, the test and trial basis must be chosen to be 
conforming and well-conditioned. With these two properties, one can guarantee 
that the condition number remains bounded. Third, after FOP is applied, the 
matrix and right-hand side in the linear system need to be computed with high 
accuracy. As discussed in Section~\ref{chap:fem}, this can be guaranteed by 
requiring that the preconditioned operator $\tilde{\mathcal{L}}$ is continuous. 

We have highlighted two classical applications (polynomial interpolation and 
spectral methods) that can be regarded as an instance of FOP. We believe that 
many other high-accuracy algorithms (existing and forthcoming)  could also be 
understood as a form of FOP, and that much can be learned by revisiting existing 
methods and establishing connections from this perspective. 

It is important to point out the potential drawbacks of FOP. First, some 
desirable structures in the unpreconditioned system may be lost. For example, 
the FOP in Section \ref{chap:fem} results in dense matrices. This negates the 
sparsity of the unpreconditioned system, one of the typical benefits of FEM. 
Another drawback is the difficulty of finding a good FOP, that is, identifying 
an operator verifying the first property listed above. Fortunately, this is 
also needed in operator preconditioning \cite{hiptmair}. Therefore, investigations 
of such operators have already taken place in the literature, see for example 
\cite{wathen, GrS06, GrS07, GSU21} and the references therein. Building on this 
well-established knowledge about operator preconditioning, it remains to come 
up with a suitable discretization.

For many applications of scientific computing, this is an open challenge. 
By overcoming it, one would obtain solutions with unprecedented accuracy.

\bibliographystyle{plain}
\bibliography{literature_plain}

\end{document}